\documentclass[a4paper,reqno]{amsart}
	\usepackage[latin1]{inputenc}
	\usepackage{xspace}
	\usepackage{hyperref}
	\usepackage{amssymb,amsmath,amsthm}
	\usepackage[shortalphabetic]{amsrefs}
\usepackage{amsthm}

\newtheorem{theorem}{Theorem}[section]
\newtheorem{lemma}[theorem]{Lemma}
\newtheorem{proposition}[theorem]{Proposition}
\newtheorem{corollary}[theorem]{Corollary}

\newtheorem{problem}[theorem]{Problem}

\theoremstyle{definition}
\newtheorem{definition}[theorem]{Definition}

\newtheorem{example}[theorem]{Example}

\theoremstyle{remark}
\newtheorem{remark}[theorem]{Remark}

\newcommand{\N}{\mathbb{N}}
\newcommand{\Z}{\mathbb{Z}}

\newcommand{\R}{\mathbb{R}}
\newcommand{\C}{\mathbb{C}}

\newcommand{\ra}{\rightarrow}

\newcommand{\lra}{\longrightarrow}

\DeclareMathOperator{\rk}{rk}	%rank
\DeclareMathOperator{\Aut}{Aut}	%automorphisms
\DeclareMathOperator{\Diff}{Diff}	%diffeomorphisms
\DeclareMathOperator{\id}{id}	%identity
\DeclareMathOperator{\Crit}{Crit}	%critical points

\newcommand{\set}[1]{\left\{ #1 \right\}}

\newcommand{\scp}[1]{\left\langle { #1 } \right\rangle}

\newcommand{\inv}{^{-1}}

\newcommand{\del}{\partial}

\newcommand{\nubar}{\bar{\nu}} %closed tubular nbhd

\usepackage{xspace}

\newcommand{\wrt}{with respect to\xspace}

\newcommand{\nbhd}{neighborhood\xspace}
\newcommand{\nbhds}{neighborhoods\xspace}

	\usepackage[pdftex]{graphicx}
	\usepackage[all]{xy}

	\usepackage{color}
	
	\usepackage{scrtime}
\numberwithin{equation}{section}
\newcommand{\g}{\gamma}
\newcommand{\G}{\Gamma}
\renewcommand{\S}{\Sigma}

\renewcommand{\Crit}{\mathcal{C}}
\newcommand{\M}{\mathcal{M}}
\newcommand{\mcR}{\mathcal{R}}
	\newcommand{\RR}{\mathcal{R}}
\newcommand{\mcS}{\mathcal{S}}
\newcommand{\SD}{\mathfrak{S}}

\newcommand{\CP}{\mathbb{C}P^2}
\newcommand{\CPbar}{\overline{\mathbb{C}P^2}}

\newcommand{\scc}{simple closed curve\xspace}
\newcommand{\sccs}{simple closed curves\xspace}

\newcommand{\swf}{simple wrinkled fibration\xspace}

\newcommand{\swfs}{simple wrinkled fibrations\xspace}
\newcommand{\Swfs}{Simple wrinkled fibrations\xspace}

\newcommand{\delp}{\del_+}
\newcommand{\delm}{\del_-}
\newcommand{\delpm}{\del_\pm}

\begin{document}

% --> HEADER --> %
\title{On 4-manifolds, folds and cusps}
\author{Stefan Behrens}
\address{Max-Planck-Institute for Mathematics\\Bonn, Germany}
\date{\today}
\keywords{4-manifolds, fold, cusp, simplified purely wrinkled fibration, broken Lefschetz fibration, surface diagram}

\begin{abstract}
We study \emph{simple wrinkled fibrations}, a variation of the simplified purely wrinkled fibrations introduced in~\cite{Williams1}, and their combinatorial description in terms of \emph{surface diagrams}.
We show that \swfs induce handle decompositions on their total spaces which are very similar to those obtained from Lefschetz fibrations. 
The handle decompositions turn out to be closely related to surface diagrams and we use this relationship to interpret some cut-and-paste operations on 4-manifolds in terms of surface diagrams. This, in turn, allows us classify all closed 4-manifolds which admit \swfs of genus one, the lowest possible fiber genus.
\end{abstract}
\maketitle

% --> MAIN TEXT -->

\section{Introduction}

After the pioneering work of Donaldson and Gompf on symplectic manifolds and Lefschetz fibrations~\cites{Donaldson,GS} (and later Auroux, Donaldson and Katzarkov on near-symplectic manifolds~\cite{ADK}), the study of singular fibration structures on smooth 4-manifolds has drawn a considerable interest among 4-manifold theorists. 
Among the highlights in the field have been existence results for so called \emph{broken Lefschetz fibrations} over the 2-sphere on all closed, oriented 4-manifolds~\cites{Akbulut-Karakurt,Baykur1,Gay-Kirby,Lekili} as well as a classification of these maps up to homotopy~\cites{Lekili,Williams1}.
Furthermore, the classical observation that Lefschetz fibrations over the 2-sphere are accessible via handlebody theory and can be described more or less combinatorially in terms of collections of simple closed curves on a regular fiber known as the \emph{vanishing cycles}~\cites{Kas,GS} was extended to the broken Lefschetz setting in~\cite{Baykur2}.

Our starting point is the work of Williams~\cite{Williams1} who introduced the closely related notion of \emph{simplified purely wrinkled fibrations}, proved their existence and exhibited a similar combinatorial description of these maps, again by collections of simple closed curves on a regular fiber, which he calls \emph{surface diagrams}.
In particular, it follows that all smooth, closed, oriented 4-manifolds can be described by a surface diagram.
However, the correspondence between simplified purely wrinkled fibrations and surface diagrams has been somewhat unsatisfactory in that it usually involves arguments using broken Lefschetz fibrations and one has to assume the fiber genus to be sufficiently high.

It is one of our goals to provide a detailed and intrinsic account of this correspondence and to clarify the situation in the lower genus cases. 
After that we will give some applications.
Let us describe the contents of this paper in more detail.

\smallskip
We begin by recalling some preliminaries from the singularity theory of smooth maps and the theory of mapping class groups of surfaces. 
This section is slightly lengthy because we intend to use it as a reference for future work.

The following two sections form the technical core of this paper.
In Section~\ref{S:SWFs over general base} we introduce \emph{\swfs} over a general base surface. 
In the case when the base is the 2-sphere our definition is almost equivalent to Williams' simplified purely wrinkled fibrations and our reason for introducing a new name is mainly to reduce the number of syllables.
We then explain how the study of \swfs reduces to certain fibrations over the annulus which we call \emph{annular \swfs}. From these we extract \emph{twisted surface diagram} and establish a correspondence between annular \swfs and twisted surface diagrams (Theorem~\ref{T:annular SWFs <-> twisted SDs}) up to suitable notions of equivalence.
Along the way we show that annular \swfs induce (relative) handle decompositions of their total spaces which are, in fact, encoded in a twisted surface diagram (Section~\ref{S:handle decompositions}).
These handle decompositions bare a very close resemblance with those obtained from Lefschetz fibrations, the only difference appearing in the framings of certain 2-handles.
The section ends with an investigation of the ambiguities for gluing surface bundles to the boundary components of annular \swfs.

In Section~\ref{S:SWFs over disk and sphere} we specialize to the case when the base surface is either the disk or the 2-sphere and recover Williams' setting. 
Using our results about annular \swfs we obtain a precise correspondence between Williams' (untwisted) surface diagrams and certain \swfs over the disk (Proposition~\ref{T:SDs and descending SWFs}).
In particular, our approach provides a direct way to construct a \swf from a given surface diagram circumventing the previously necessary detour via broken Lefschetz fibrations.\footnote{By now this can be considered as a special case of~\cite{GK3} which appeared while we were writing this paper.}

Next, we address the question which surface diagrams describe \swfs that extend to fibrations over the sphere and thus describe closed 4-manifolds. Just as in the theory of Lefschetz fibrations the key is to understand the boundary of the associated \swf over the disk. We show how to identify this boundary with a mapping torus and describe its monodromy in terms of the surface diagram.
Unfortunately, it turns out that the boundary is much harder to understand than in the Lefschetz setting.

We then go on to review the handle decompositions exhibited in Section~\ref{S:SWFs over general base} when the base is the disk or the sphere and describe a recipe for drawing Kirby diagrams for them. 
To complete the picture, we compare our decompositions with the ones obtained via simplified broken Lefschetz fibrations.

In the Sections~\ref{S:substitutions} and~\ref{S:genus 1 classification} we give some applications.
We show that certain substitutions of curve configurations in surface diagrams correspond to cut-and-paste operations on 4-manifolds. 
In particular, we give a surface diagram interpretation of blow-ups and sum-stabilizations, i.e.~connected sums with~$\CP$,~$\CPbar$ and~$S^2\times S^2$. 
Using these we easily obtain a classification of closed 4-manifolds which admit \swfs with the lowest possible fiber genus.
\begin{theorem}\label{T:genus 1 classification, intro}
A smooth, closed, oriented 4-manifold admits a \swf of genus one if and only if it is diffeomorphic to $k S^2\times S^2$ or $m\CP\# n\CPbar$ where~$k,m,n\geq1$.
\end{theorem}
This result should be compared to~\cite{Baykur-Kamada} and~\cite{Hayano1} where a classification of genus one simplified broken Lefschetz fibration is addressed but only partially achieved. 
However, it should also be noted that the latter class of maps is strictly larger than that of genus one simple wrinkled fibrations and it is thus conceivable that the classification is more complicated.

In the final Section~\ref{S:concluding remarks} we close this paper by highlighting what we consider as some of the main problems in the field and by outlining some related developments.

\subsection*{Conventions}
By default all manifolds are smooth, compact and orientable and all diffeomorphisms are orientation preserving. 
When we speak of \nbhds of submanifolds we always mean tubular \nbhds. We use the symbol~$\nu S$ (resp.~$\nubar S$) for an open (resp.~closed) tubular \nbhd of a submanifold~$S\subset M$.

For induced orientations on boundaries we use the \emph{outward normal first} convention.
Moreover, if $f\colon M\ra N$ is smooth, $M$ and~$N$ are connected and~$p\in N$ is a regular value, then orientations on two out of the three manifolds~$M$,~$N$ and~$f\inv(p)$ induce an orientation on the third as follows. 
There is a small ball~$D\subset N$ containing~$p$ such that~$f\inv(D)$ can naturally be identified with~$f\inv(p)\times D$ and we choose the third orientation such that this identification preserves orientations where~$f\inv(p)\times D$ carries the product orientation.

Finally, (co-)homology is always taken with integral coefficients.
Exceptions to these rules will be explicitly stated and we reserve the right to sometimes restate some of the conditions for emphasis.

\subsection*{Acknowledgements}
This work is part of the author's ongoing PhD project carried out at the Max-Planck-Institute for Mathematics in Bonn, Germany. 
The author would like to thank Inanc Baykur for helpful comments on an early draft of this paper as well as his advisor Prof.~Dr.~Peter Teichner for letting him work on this project.
The author is supported by an IMPRS Scholarship of the Max-Planck-Society.

\section{Preliminaries}
	\label{S:preliminaries}
To fix some terminology, let $f\colon M\ra N$ be a smooth map with differential $df\colon TM\ra TN$. A \emph{critical point} (or a \emph{singularity}) of~$f$ is a point~$p\in M$ such that~$df_p$ is not surjective. The set of critical points, called the \emph{critical locus} of $f$, will be denoted by 
\begin{equation*}
	\Crit_{f}:=\left\{p\in M\middle\vert \rk df_p < \dim N\right\} \subset M.
\end{equation*}
The image of a critical point is called a \emph{critical value} and the set of all critical values is called the \emph{critical image} of~$f$.

As customary, we will call the preimage of a point a \emph{fiber}, usually decorated with the adjectives regular or singular indicating whether or not the fiber contains critical points. Note that regular fibers are always smooth submanifolds with trivial normal bundle.

\subsection{Folds, cusps and Lefschetz singularities}
	\label{S:singularity theory}
As a warm up, 
recall that a generic map from any compact manifold to a 1-dimensional manifold has only finitely many critical points on which it is injective and, moreover, all critical points are of \emph{Morse type}, i.e. they are locally modeled\footnote{A map $f\colon M^m\ra N^n$ is \emph{locally modeled around $p\in M$} on $f_0\colon\colon\R^m\ra \R^n$ if there are local coordinates around~$p$ and~$f(p)$ mapping these points to the origin such that the coordinate representation of~$f$ agrees with~$f_0$.} 
on the maps
	\begin{equation*}
	(x_1,\dots,x_n) \mapsto -x_1^2-\dots-x_k^2+x_{k+1}^2+\dots+x_n^2,
	\end{equation*}
where the number~$k$ is called the \emph{(Morse) index} of the critical point. 

A similar statement holds for maps to surfaces. For convenience we will take the source to be 4-dimensional from now on.
In this setting the Morse critical points are replaced by two types of singularities known as \emph{folds} and \emph{cusps} which can also be described in terms of local models. 
The model for a fold singularity is the map $\R^4\ra \R^2$ given by the formula 
\begin{equation}\label{E:fold model}
	(t,x,y,z)\mapsto (t,-x^2-y^2\pm z^2)
\end{equation}
and the cusps are locally modeled on
\begin{equation}\label{E:cuspmodel}
	(t,x,y,z)\mapsto (t,-x^3 +3tx-y^2\pm z^2).
\end{equation}
If the sign in either of the above equations is positive (resp.\ negative), then the singularity is called \emph{indefinite} (resp.\ \emph{definite}). 

An easy calculation shows that the critical loci of the fold and cusp models are given by $\left\{\, (r,0,0,0) \;\middle\vert\; r\in\R \,\right\}$ and $\left\{\, (r^2,r,0,0) \;\middle\vert\; r\in\R \,\right\}$, respectively.
As a consequence, the critical image of a smooth map is a smooth 1-dimensional submanifold near fold and cusp points.
The critical images of both models are shown in Figure~\ref{F:fold and cusp}.
	\begin{figure}[h]
	\label{F:fold and cusp}
	\includegraphics{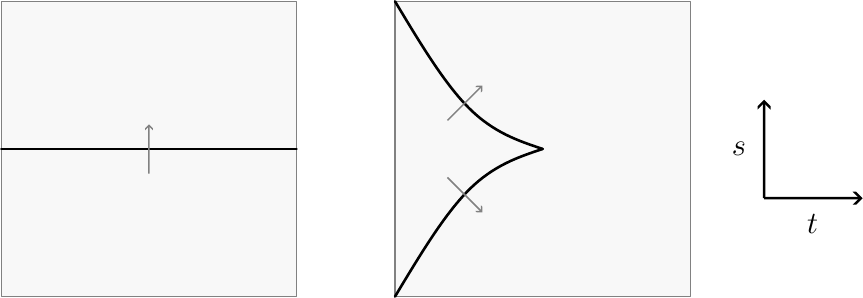}
	\caption{The critical images of the fold and cusp models.}
	\end{figure}
Note that the critical image is smoothly embedded in the fold model where as in the cusp case it is topologically embedded via a smooth homeomorphism whose inverse fails to be smooth only at the cusp point.

It follows directly from the models that folds always come in 1-dimensional families on which the map restricts to an immersion. We will usually be sloppy and refer to such an arc of fold points in the source as well as their image in the target as \emph{fold arcs}. 
Furthermore, cusps are isolated in the critical locus in the sense that there is a small \nbhd which contains no other cusps. However, cusps are not isolated singularities. 
In fact, one can show that any cusp is surrounded by two fold arcs, at least one of which is indefinite. 

We can now state the normal form of generic maps from 4-manifolds to surfaces.

\begin{theorem}[Normal form of maps to surfaces]\label{T:generic maps to surfaces}
A generic map from a 4-manifold to a surface has only fold and cusp singularities, it is injective on the cusps and restricts to an immersion with only transverse intersections between fold arcs.
\end{theorem}
Note, in particular, that the above discussion shows that the critical locus of a generic map to a surface is a smooth 1-dimensional submanifold of the source.
For more details, including a proof of the above theorem for arbitrary source dimension, we refer the reader to~\cite{GG}. 
\begin{remark}\label{R:Morse 2-functions}
Recently, these generic maps to surfaces have appeared under the name \emph{Morse 2-functions} in the work of Gay and Kirby~\cites{GK1,GK2,GK3}. 
\end{remark}

In what follows we will only deal with indefinite singularities. So from now on, when we speak of folds and cusps, we will always mean the indefinite ones.

\smallskip
Figure~\ref{F:fold and cusp} contains some further decorations which we will now explain. 
Both, the fold and the cusp singularity are intimately related to 3-dimensional Morse-Cerf theory. 
The fold models a trivial homotopy of a Morse functions with one critical point (of index two) on the vertical slices. This means that the model restricted to a small arc transverse to the fold locus is a Morse function with one critical point of index one or two depending on the direction. The arrows in the picture indicate the direction in which the index is two. 
Note that the topology of the fibers of either side of a fold arc is necessarily different.

Similarly, the cusp is also a homotopy of Morse functions on the vertical slices, although a nontrivial one. It models the cancellation of a pair of critical points (of index one and two). 
The arrows indicate the index two direction of the fold arcs adjacent to the cusp.

\smallskip
For the moment, this is all we have to say about folds and cusps. 
Another important type of singularity which has its roots in (complex) algebraic geometry is the \emph{Lefschetz singularity} and its local model is given in complex coordinates by
\begin{equation*}
	L \colon \C^2 \ra \C
	\quad ; \quad
	(z,w)\mapsto zw.
\end{equation*}
At this point it becomes important whether the charts that we use to model the map are orientation preserving. Although this does not matter for folds and cusps\footnote{For both models there are orientation reversing diffeomorphisms which leave the map invariant}, it makes a surprisingly big difference in the case of Lefschetz singularities. So from now on we will always use orientation preserving charts to model singularities whenever the source or target are oriented.

\medskip
As stated in the introduction, maps with (indefinite) fold, cusp and Lefschetz singularities have been prominently featured in many research papers over the past decade. Unfortunately, various authors have used various names for various types of maps and there is yet no commonly accepted terminology in the field.
For the purpose of this paper we will use the following terminology.
\begin{definition}\label{D:singular fibrations}
A surjective map $f\colon X\ra B$ from an oriented 4-manifold to an oriented surface is called 
(a) a \emph{wrinkled fibration}, 
(b) a \emph{(broken) Lefschetz fibration} or 
(c) a \emph{broken fibration}
if its critical locus contains only
\begin{enumerate}
	\item[(a)] indefinite folds and cusps,
	\item[(b)] Lefschetz singularities (and indefinite folds),
	\item[(c)] indefinite folds, cusps and Lefschetz singularities,
\end{enumerate}
all critical points are contained in the interior of~$X$ and all intersections in the critical image are transverse intersections of fold arcs.
\end{definition}
In accordance with the use of the word fibration we will usually refer to the source as the \emph{total space} and to the target as the \emph{base}.
Note that the regular fibers of a broken fibration are (orientable) surfaces. Furthermore, if we assume that $\del X= f\inv(\del B)$, which we will do later on, then the fibers are closed. 

\smallskip
It is quite useful to think of broken fibrations as (singular) families of surfaces parametrized by the base. 
More precisely, the images of the folds and cusps cut the base into several regions which may or may not contain Lefschetz singularities. The regular fibers are (orientable) surfaces whose topological type depends only on the region that it maps into.
One thus decorates the base with the topological type of the fibers over each region together with some information about what happens to a fiber if one crosses a fold arc (the little arrows we have indicated above together with the corresponding fold vanishing cycle) or runs into a Lefschetz singularity (the Lefschetz vanishing cycle). Under certain circumstances this data is enough to determine the map as we will see later on (see also~\cite{GK3}).

\medskip
We finish this section with a short review of the homotopy classification of broken fibrations over~$S^2$ that was mentioned in the introduction. 
An important contribution of Lekili~\cite{Lekili} is that he showed how to pass back and forth between broken Lefschetz fibrations and wrinkled fibrations via two \emph{local homotopies}, i.e.~homotopies that are supported in arbitrarily small balls.
As portrayed in Figure~\ref{F:Lefschetz vs cusp} one can \emph{wrinkle} a Lefschetz point into an indefinite triangle (i.e. an indefinite circle with three cusps) and one can exchange a cusp for a Lefschetz singularity, this move is sometimes called \emph{unsinking} a Lefschetz point from a fold. 
(Moreover, he showed that these modifications work equally well with achiral Lefschetz singularities which, together with the results of~\cite{Gay-Kirby}, proves the existence of broken Lefschetz fibrations.) 
	\begin{figure}
	\includegraphics{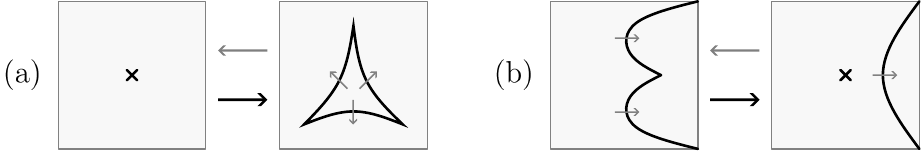}
	\caption{(a) Wrinkling and (b) unsinking a Lefschetz singularity.}
	\label{F:Lefschetz vs cusp}
	\end{figure}
As a consequence, one can translate questions about broken fibrations into questions about wrinkled fibrations which are accessible by means of singularity theory. 
For example, there is a structural result similar to Theorem~\ref{T:generic maps to surfaces} for generic homotopies between wrinkled fibrations. 
The basic building blocks include isotopies of the base and total space and three types of modifications (and their inverses) that are realized by local homotopies:
the \emph{birth/death}, the \emph{merge} and the \emph{flip}. 
Figure~\ref{F:basic homotopies} shows their effect on the critical image. 
	\begin{figure}[h]
	\includegraphics{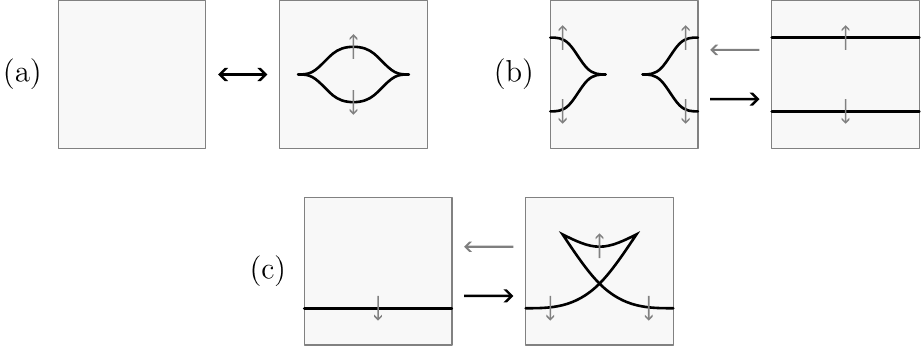}
	\caption{The basic local homotopies: (a) birth, (b) merge, (c) flip.}
	\label{F:basic homotopies}
	\end{figure}
In general, such a generic homotopy will pass through maps with definite singularities. 
However, the main theorem in~\cite{Williams1} states that indefinite singularities can, in fact, be avoided. In other words, any two homotopic wrinkled fibrations are homotopic through wrinkled fibrations.

\begin{remark}\label{R:reversing homotopy moves}
It has become common to refer to an application of any of the above mentioned modifications as \emph{moves} performed on a broken fibration.
It is important to note that most of these moves are not strictly reversible in the following sense. If the critical image of a given broken fibration exhibits a configuration as on the left hand side of any of the pictures, then it is always possible to replace it by the configuration on the right hand side. However, it might not be possible to go into the other direction. The only exception is the birth. In all other cases some extra conditions are needed to go from right to left. This is indicated in our pictures with shaded arrows.
\end{remark}
\begin{remark}\label{R:merging folds and cusps}
There has been some disagreement in the literature about which direction in Figure~\ref{F:basic homotopies}(b) should be called merge and which inverse merge. 
To avoid this decision we will simply speak of \emph{merging cusps} and \emph{merging folds}, respectively. 
\end{remark}

\subsection{Surfaces and simple closed curves}
	\label{S:mapping class groups review}

As we pointed out, the regular fibers of broken fibrations are surfaces and these fibers will be prominently featured later on. 
Unfortunately, this is yet another field of mathematics in which different authors use different conventions and, in the current author's experience, it can be confusing to decide whether a statement in some reference actually applies to a situation at hand.
For that reason we will give very precise definitions, deliberately risking to be overly precise.

\smallskip
By a \emph{surface}~$\S$ we mean a compact, orientable, 2-dimensional manifold, possibly with boundary and some marked points in the interior.
A \emph{simple closed curve} in~$\S$ is a closed, connected, 1-dimensional submanifold of~$\S$ that does not meet the boundary or the marked points. 
We usually consider \sccs up to ambient isotopy in~$\S$ relative to~$\del\S$ and the marked points and will not make a notational distinction between a \scc and its isotopy class.
Note that according our definition \sccs are unoriented objects. However, from time to time it will be convenient to choose orientations on them in order to speak of their homology classes.

\smallskip
Given two \sccs $a,b\subset\S$ we define their \emph{geometric intersection number} as
\begin{equation*}
	i(a,b):=\min\left\{ \#(\alpha\cap\beta) \middle\vert \alpha\sim a,\; \beta\sim b,\; \alpha\pitchfork \beta \right\}\in\N
\end{equation*}
where the signs $\sim$ and $\pitchfork$ indicate isotopy and transverse intersection. 
If the curves as well as the surface are oriented, then we also have an \emph{algebraic intersection number} which is obtained by a signed count of intersections after making the curves transverse. Equivalently, this number can be described as
	\begin{equation*}
	\scp{a,b}:=\scp{[a],[b]}_\S:=\scp{[a],[b]}_{H_1(\S)}\in\Z
	\end{equation*}
where bracket on the right hand side denotes the intersection form on~$H_1(\S)$.

Note that the algebraic intersection number is alternating and depends only on the homology classes of the oriented \sccs while the geometric intersection number is symmetric and depends on the isotopy classes.
Both intersection numbers have the same parity (i.e.~even or odd) and satisfy the inequality
	\begin{equation}\label{E:algebraic vs geometric intersections}
	|\scp{a,b}|\leq i(a,b).
	\end{equation}
We say that $a$ and $b$ are \emph{geometrically dual} (resp.\ \emph{algebraically dual}) if their geometric (reap.\ algebraic) intersection number is one.

\smallskip
A \scc $a\subset\S$ is called \emph{non-separating} if its complement is connected, otherwise it is called \emph{separating}. 
Note that a \scc is separating if and only if it is null-homologous (with either orientation) and thus \sccs that have geometric or algebraic duals are automatically non-separating.

\subsubsection{Diffeomorphisms of surfaces}
Let us now turn to diffeomorphisms of surfaces. 
Let $\Diff^+(\S,\del\S)$ denote the set of orientation preserving diffeomorphisms that restrict to the identity on~$\del\S$ and preserve the set of marked points. 
The \emph{mapping class group} of~$\S$ is defined as 
	\begin{equation*}
		\M(\Sigma):=\pi_0(\Diff^+(\S,\del\S),\id).
	\end{equation*}
Given a \scc $a\subset\S$ there is a well defined mapping class $\tau_a\in\M(\S)$ called the (right-handed) \emph{Dehn twist} about~$a$. 
Similarly, any simple arc~$r\subset\S$ that connects two distinct marked points gives rise to a \emph{half twist}~$\bar{\tau}_r\in\M(\S)$.

It is well known that $\M(\S)$ is generated by the collection of Dehn twist and half twists, where the latter are only needed in the presence of marked points.
On the other hand, mapping classes can be effectively studied by their action on (isotopy classes of) simple closed curves. 
In particular, it is desirable to understand the effect of Dehn twists on \sccs. While this can be tricky, the situation simplifies significantly on the level of homology classes.
\begin{proposition}[Picard-Lefschetz formula]\label{T:Picard-Lefschetz formula}
Let $\S$ be a surface, $a\subset\S$ a \scc and let $x\in H_1(\S)$. Then for any orientation on $a$ we have
\begin{equation}
(\tau_a^k)_*x=x+k\scp{[a],x}[a].
\end{equation}
In particular, if~$b$ is an oriented \scc, then
\begin{equation}
[\tau_a^k(b)]=[b]+k\scp{[a],[b]}[a].
\end{equation}
\end{proposition}
\begin{proof}
See \cite{primer}, Proposition~6.3.
\end{proof}
\begin{remark}\label{R:Picard-Lefschetz on torus}
The Picard-Lefschetz formula is particularly useful for the torus since, in that case, mapping classes are completely determined by their action on homology.
\end{remark}

Another useful tool is the so called \emph{change of coordinates principle} which roughly states that any two configurations of \sccs on a surface with the same intersection pattern can be mapped onto each other by a diffeomorphism. We will only use the following special cases. For details we refer to~\cite{primer}, Chapter~1.3.
\begin{proposition}[Change of coordinates principle]\label{T:change of coordinates principle}
If~$a,b\subset\S$ is a pair of non-separating \sccs, then there exists some $\phi\in\Diff^+(\S,\del\S)$ such that~$\phi(a)=b$.
Furthermore, if~$a,b$ and~$a',b'$ are two pairs of geometrically dual curves, then there is some  $\phi\in\Diff^+(\S,\del\S)$ such that~$\phi(a)=a'$ and~$\phi(b)=b'$.
\end{proposition}

\subsubsection{Mapping tori and their automorphisms}
	\label{S:mapping tori}
Given a surface~$\S$ and a diffeomorphism $\mu\colon \S\ra \S$ we can form its \emph{mapping torus}
	\begin{equation*}
	\S(\mu) := \big( \S\times[0,1] \big) / \big( (x,1)\sim(\mu(x),0) \big)
	\end{equation*}
which is a 3-manifolds that carries a canonical map to~$S^1\cong [0,1]/\{0,1\}$ which turns out to be a submersion. In other words, $\S(\mu)$ fibers over~$S^1$. 
If $\S$ is oriented and $\mu$ is orientation preserving, then our conventions in the introduction induce an orientation on~$\S(\mu)$.
It is well known that all surface bundles over~$S^1$ can be described as mapping tori. Indeed, if a 3-manifold fibers over~$S^1$, then one chooses a fiber and a lift of a vector field that determines the orientation of~$S^1$ and the return map of the flow of this vector field induces a diffeomorphism of the fiber which is usually called the \emph{monodromy}.

\smallskip
Let $Y$ be an oriented 3-manifold that fibers over the circle via a map~$f\colon Y\ra S^1$. 
An \emph{automorphism} of $(Y,f)$ is an orientation and fiber preserving diffeomorphism of~$Y$. We denote the group of automorphisms by~$\Aut(Y,f)$ or simply by~$\Aut(Y)$ when the fibration is clear from the context.
If we identify~$Y$ with a mapping torus, say~$\S(\mu)$, then we obtain a description of~$\Aut(Y)$ in terms of diffeomorphisms of~$\S$.
Indeed, any element~$\phi\in\Aut(\S(\mu))$ can be considered as a path~$(\phi_t)_{t\in[0,1]}$ in~$\Diff^+(\S)$ connecting some element~$\phi_0\in\Diff^+(\S)$ to~$\phi_1=\mu\inv\phi_0\mu$. 
In particular, $\phi_0$~must be isotopic to~$\mu\inv\phi_0\mu$ and thus represents an element of~$C_{\M(\S)}(\mu)$, the centralizer in~$\M(\S)$ of (the mapping class represented by)~$\mu$. 
Elaborating on this idea one arrives at the conclusion that 
	\begin{equation}\label{E:automorphisms of mapping tori}
	\pi_0\big(\Aut(Y)\big) \cong \pi_0\big(\Aut(\S(\mu))\big) \cong 
	C_{\M(\S)}(\mu) \ltimes \pi_1(\Diff(\S),\id),
	\end{equation}
where the multiplication on the right hand side is given by 
	\begin{equation*}
	(g,\sigma)\cdot(h,\tau)=(h\circ g, (g\inv\tau g)\ast\sigma ).
	\end{equation*}
This means that there are essentially two types of automorphism of mapping tori, the ones that are constant on the fibers coming from $C_{\M(\S)}(\mu)$ and the ones coming from~$\pi_1(\Diff(\S),\id)$ that vary with the fibers and restrict to the identity on the reference fiber.
Fortunately, there are no non-constant automorphisms most of the time due to the following classical result.
\begin{theorem}[Earle-Eells,~\cite{Earle-Eells}]\label{T:Earle-Eells}
Let $\S$ be a closed, orientable surface of genus~$g$ without marked points. Then
	\begin{equation*}
	\pi_1(\Diff(\S),\id)\cong
	\begin{cases}
	\Z_2 & \text{if $g=0$} \\
	\Z\oplus\Z & \text{if $g=1$} \\
	1 & \text{if $g\geq2$}.
	\end{cases}
	\end{equation*}
\end{theorem}
Hence, as soon as the genus of the fiber of a mapping torus is at least two, all automorphisms are isotopic (through automorphisms) to constant ones.

\begin{remark}\label{R:Aut ain't Diff!}
It is important not to confuse the group~$\Aut(Y)$ with the group of all (orientation preserving) diffeomorphisms of~$Y$. A general diffeomorphism will not even be isotopic to a fiber preserving one!
\end{remark}

Theorem~\ref{T:Earle-Eells} has many important consequences of which we only highlight one.
\begin{corollary}\label{T:surface bundles over the sphere}
Let $P\ra S^2$ be a surface bundle with closed fibers of genus~$g$.
\begin{enumerate}
	\item If~$g=0$, then~$P$ is diffeomorphic to $S^2\times S^2$ or $\CP\#\CPbar$.
	\item If~$g=1$, then~$P$ is diffeomorphic to $T^2\times S^2$, $S^1\times S^3$ or $S^1\times L(n,1)$.
	\item If~$g\geq2$, then~$P$ is diffeomorphic to $\S_g\times S^2$	
\end{enumerate}
\end{corollary}
\begin{proof}
For the genus one case see~\cite[Lemma~10]{Baykur-Kamada}. The other cases are well known.
\end{proof}

\section{\Swfs over general base surfaces}
	\label{S:SWFs over general base}
We are finally ready to introduce the main objects of study in this paper.

\begin{definition}\label{D:simple wrinkled fibrations}
Let $X$ be a 4-manifold and $B$ a surface, both oriented. A \emph{\swf} with \emph{total space}~$X$ and \emph{base}~$B$ is a surjective smooth map of pairs $w\colon (X,\del X)\ra (B,\del B)$ with the following properties:
\begin{enumerate}
	\item $w$ is a wrinkled fibration, i.e.~$\Crit_{w}$ contains only indefinite folds and cusps,
	\item $\Crit_{w}\cap\del X=\emptyset$,
	\item $\Crit_{w}$ is non-empty, connected, and contains a cusp,
	\item $w$ is injective on $\Crit_{w}$ and
	\item all fibers of $w$ are connected.
\end{enumerate}
Two \swfs $w\colon X\ra B$ and $w'\colon X'\ra B'$ are \emph{equivalent} if 
there are orientation preserving diffeomorphisms~$\hat{\phi}\colon X\ra X'$ and~$\check{\phi}\colon B\ra B'$ such that~$w'\circ\hat{\phi}=\check{\phi}\circ w$.
\end{definition}

Since we assume the base and total space of a \swf to be oriented, the regular fibers are closed, oriented surfaces (of varying genus as explained below). We can thus define the \emph{genus} of~$w$ as the maximal genus among all regular fibers.
A \nbhd of the critical image of a \swf is shown in Figure~\ref{F:SWF base}
	\begin{figure}[h]
	\includegraphics{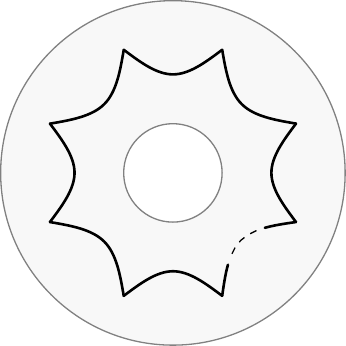}
	\caption{A \nbhd of the critical image of a \swf.}
	\label{F:SWF base}
	\end{figure}

\smallskip
Before we continue we make some remarks about the definition.
\begin{remark}\label{R:difference from SPWFs}
\Swfs over $S^2$ are essentially the same as Williams' simplified purely wrinkled fibrations with two minor differences. One one hand we do not put restrictions on the fiber genus but on the other we require the presence of cusps. 
Both conditions can always be achieved by applying 
a \emph{flip-and-slip move} (see Remark~\ref{R:flip+slip} below)
and are thus merely of technical nature. Moreover, the ``\swfs without cusps'' are easily classified (see Example~\ref{eg:ADK sphere}) so that one does not lose too much by ignoring them.
\end{remark}

\begin{remark}\label{R:flip+slip}
Given a \swf over~$S^2$ there is an important homotopy to another such \swf which has become known as a \emph{flip-and-slip move}. Its effect on the base diagram is shown in Figure~\ref{F:flip+slip}. 
	\begin{figure}
	\includegraphics{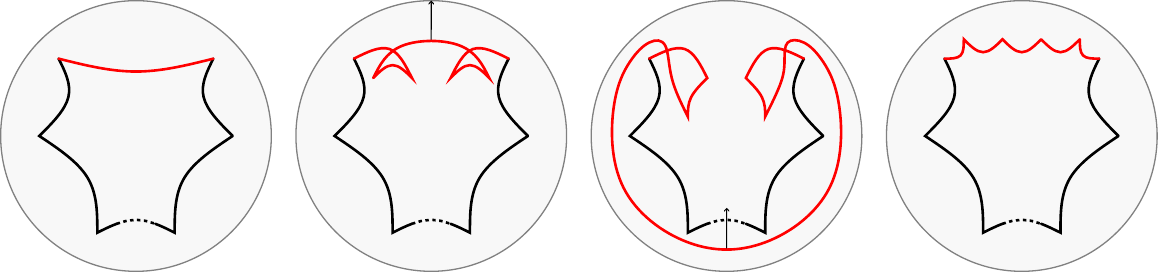}
	\caption{The base diagrams during a flip-and-slip move. (The pictures show the complement of a disk in the lower genus region of the original fibration.)}
	\label{F:flip+slip}
	\end{figure}
One first perform two flips on the same fold arc and then chooses an isotopy of the total space (the \emph{slip}) during which the critical image undergoes the changes demonstrated in the picture.
A flip-and-slip increases the fiber genus by one and introduces four new cusps.
\end{remark}

\begin{remark}\label{R:SWFs are simple}
In spite of the lengthy definition, \swfs are arguably the simplest possible maps from 4-manifolds to surfaces, at least as far as their singularity structure is concerned. As will be explained in detail it is this simplicity which makes it possible to give nice combinatorial descriptions of 4-manifolds.
\end{remark}

\begin{remark}\label{R:homotopy vs. equivalence}
So far \swfs have usually been studied up to homotopy instead of equivalence. 
However, we believe that the former point of view does not interact well with surface diagrams (which will be introduced momentarily) while the latter fits in perfectly.
It would be interesting to relate the concepts of homotopy and equivalence but to our knowledge there is no obvious way to do so.
\end{remark}

Given the rather specialized nature of \swfs one might wonder whether they actually exist. This is indeed the case and we begin by giving some simple constructions.
\begin{example}[Surface bundles]\label{eg:birth on bundles}
Let $\pi\colon X\ra B$ be a surface bundle over a surface~$B$ with closed fibers of genus~$g$. Then we can perform a birth homotopy on~$\pi$ to obtain a genus~$g+1$ \swf with two cusps. 
\end{example}
\begin{example}[Lefschetz fibrations]\label{eg:wrinkled Lefschetz fibrations}
If $f\colon X\ra B$ is a Lefschetz fibration (possibly achiral) with closed fibers of genus~$g$, then after wrinkling all the Lefschetz singularities we obtain a number of disjoint circles with three cusps in the critical image. 
By suitably merging cusps we can turn this configuration into a single circle resulting in a \swf of genus~$g+1$. 
\end{example}

\begin{example}[The case without cusps]\label{eg:ADK sphere}
This example includes the broken Lefschetz fibration on~$S^4$ from~\cite{ADK} that was mentioned in the introduction.
Let~$\Omega$ be a cobordism from~$\S_g$ to~$\S_{g-1}$ together with a Morse function~$\mu\colon \Omega\ra I$ with exactly one critical point of index two. Then~$\mu\times\id \colon \Omega\times S^1 \ra I\times S^1$ is a stable map with one circle of indefinite folds which fails to be a \swf only because it does not have any cusps. 
Nevertheless, we can use~$\Omega\times S^1$ to build wrinkled fibrations over~$S^2$ by suitably filling in the two boundary components with~$\S_g\times D^2$ and~$\S_{g-1}\times D^2$ such that the fibration structures on the boundary extends. 
Using the handle decomposition constructed in~\cite{Baykur2} it is easy to see that this constructions one gives the following total spaces: $P\# S^1\times S^3$ where~$P$ is any $\S_{g-1}$-bundle over~$S^2$ and, if~$g=1$, $S^4$~and some other manifolds with finite cyclic fundamental group (see~\cites{Baykur-Kamada,Hayano1}). 
Having build these maps one can then apply a flip-and-slip to obtain honest \swfs. In particular, we see that~$S^4$ carries a \swf of genus two.

As a side remark, the above mentioned genus one fibration on~$S^4$ already appeared in~\cite{ADK} and is probably the reason why people became interested in constructing broken fibrations on general 4-manifolds.
\end{example}

The above examples show that \swfs can be considered as a common generalization of surface bundles and (achiral) Lefschetz fibrations. 
The vastness of this generalization is indicated by the following remarkable theorem.
\begin{theorem}[Williams \cite{Williams1}]\label{T:Williams existence}
Let $X$ be a closed, oriented 4-manifold. Then any map $X\ra S^2$ is homotopic to a \swf of arbitrarily high genus.
\end{theorem}
\begin{remark}\label{R:existence proof}
Williams' proof builds on results of Gay and Kirby~\cite{Gay-Kirby} which, in turn, depend on deep theorems in 3-dimensional contact topology\footnote{Eliashberg's classification of overtwisted contact structures and the Giroux correspondence between contact structures and open book decompositions}. 
This somewhat unnatural dependence could be removed by refining the singularity theory based approach of~\cite{Baykur1} to produce maps which are injective on their critical points.
\end{remark}

Williams~\cite{Williams1} also introduced a combinatorial description of \swfs over~$S^2$ in terms of what he calls \emph{surface diagrams}. In the remainder of this section we will generalize his construction to the setting of general base surfaces and prove a precise correspondence. Along the way we will see how \swfs give rise to handle decompositions. In Section~\ref{S:SWFs over disk and sphere} we will return to Williams' surface diagrams and use them to prove some results. 

\medskip
Let $w\colon X\ra B$ be a \swf. 
As explained in Section~\ref{S:singularity theory}, it follows from the definition of \swfs that the critical locus~$\Crit_w\subset X$ of a \swf~$w\colon X\ra B$ is a smoothly embedded circle and that $w$~restricts to a topological embedding of~$\Crit_w$ into~$B$. Furthermore, the critical image~$w(\Crit_w)$ separates~$B$ into two components. Indeed, if its complement were connected, then all regular fibers would be diffeomorphic. But according to the fold model, the topology of the fibers on the two sides of a fold arc must be different. In fact, since we require that all fibers are connected, the genus on one side has to be one higher than on the other side. We will call the two components of~$B\setminus w(\Crit_w)$ the \emph{higher} (resp. \emph{lower}) \emph{genus region}. 

We would like to understand more precisely how the topology of the fibers changes across the critical image. 
A \emph{reference path} for~$w$ is an oriented, embedded arc~$R\subset B$ that connects a point~$p_+$ in the higher genus region to a point~$p_-$ in the lower genus region and intersects~$w(\Crit_w)$ transversely in exactly one fold point. Then the \emph{reference fibers}~$\S_\pm(R):=w\inv(p_\pm)$ over the \emph{reference points}~$p_\pm$ are closed, oriented surfaces. 
\begin{lemma}\label{T:fold vanishing cycles}
A reference path $R\subset B$ induces a nonseparating \scc $\g(R)\subset\S_+(R)$ which depends only on the isotopy class of $R$ relative to its reference points and the cusps.
\end{lemma}
\begin{definition}\label{D:fold vanishing cycle}
The curve $\g(R)\subset\S_+(R)$ is called the \emph{(fold) vanishing cycle} associated to~$R$.
\end{definition}
\begin{proof}
The fold model implies that $w\inv(R)$ is a cobordism from $\S_+(R)$ to $\S_-(R)$ on which $w$~restricts to a Morse function with exactly one critical point of index~2. Thus $w\inv(R)$ is diffeomorphic to $\S_+(R)\times[0,1]$ with a (3-dimensional) 2-handle attached along a simple closed curve in $\S_+(R)\times\set{1}$ which is canonically identified with a \scc $\g(R)\subset\S_+(R)$.
\end{proof}

Next, let us look at what happens around the cusp. Let $R_1$~and~$R_2$ be two reference paths for~$w$ with common reference points and assume that their interiors are disjoint. We call~$R_1$ and~$R_2$ \emph{adjacent} if their union~$R_1\cup R_2$ bounds a disk in~$B$ that contains exactly one cusp.
\begin{lemma}\label{T:adjacent reference paths}
Let $R_1$ and $R_2$ be adjacent reference paths. Then the vanishing cycles $\g(R_1)$ and $\g(R_2)$ in $\S_+:=\S_+(R_1)=\S_+(R_2)$ are geometrically dual.
\end{lemma}
\begin{proof}
As in the proof of Lemma~\ref{T:fold vanishing cycles} the preimages~$w\inv(R_i)$, $i=1,2$, are both cobordisms from~$\S_+$ to~$\S_-$, each consisting of a 2-handle attachment along $\g(R_i)$.
By reversing the orientation of~$R_1$ we can consider~$w\inv(R_1)$ as a cobordism from~$\S_-$ to~$\S_+$, now consisting of a 1-handle attachment. In this process the former attaching sphere of the 2-handle~$\g(R_1)$ becomes the belt sphere of the 1-handle. 

Gluing~$w\inv(R_1)$ and~$w\inv(R_2)$ together along~$\S_+$ gives a cobordism from~$\S_-$ to itself consisting of a 1-handle attachment followed by a 2-handle attachment. 
Now recall that the cusp singularity models the death (or birth) of a canceling pair of critical points. Hence, the attaching sphere of the 2-handle, which is~$\g(R_2)$, intersects the belt sphere of the 1-handle, which is~$\g(R_1)$, in a single point.
\end{proof}

Looking a bit ahead, our strategy will be to choose suitable collections of reference paths and to study \swfs in terms of the induced collection of vanishing cycles. The only obstacle for doing so is the possibly complicated topology of the base surface. But this can easily be overcome by the following observation. We can cut the base into three pieces
	\begin{equation*}
	B=B_+\cup A\cup B_-
	\end{equation*}
where $A$ is a regular \nbhd of the critical image of~$w$ (diffeomorphic to an annulus) and~$B_\pm$ are the closures of the complement of~$A$. The subscript in~$B_\pm$ indicates whether the surface is contained in the higher or lower genus region. Note that~$w$ restricts to surface bundles over $B_\pm$ and, although complicated, these are a rather well studied class of objects. Thus the interesting new part of~$w$ is the restriction $w\inv(A)\ra A$ which is a \swf over an annulus. Moreover, this fibration has the property that the critical image does not bound a disk in~$A$ or, in other words, it is boundary parallel.
\begin{definition}\label{D:annulas SWFs}
A \swf $w\colon W\ra A$ over an annulus $A$ is called \emph{annular} if its critical image is boundary parallel.
\end{definition}
So in order to understand \swfs over any base surface, it is enough to understand annular \swfs and this is where surface diagrams enter the picture. 
The remainder of this section is devoted to proving the following theorem. 
\begin{theorem}\label{T:annular SWFs <-> twisted SDs}
There is a bijective correspondence between annular \swfs up to equivalence and twisted surface diagrams up to equivalence
\end{theorem}
We will split the proof of the theorem into the two obvious parts. The first part is the subject of Section~\ref{S:SWFs -> SDs} (see Proposition~\ref{T:annular SWFs -> twisted SDs}) and the second is treated in Section~\ref{S:SDs -> SWFs} (see Proposition~\ref{T:annular SWFs <- twisted SDs}). Along the way, we will see in Section~\ref{S:handle decompositions} that, just as Lefschetz fibrations, annular \swfs are directly accessible via handlebody theory.

\begin{remark}\label{R:Gay-Kirby reconstruction}
Recently Gay and Kirby have published a result that contains Theorem~\ref{T:annular SWFs <-> twisted SDs} as a special case~\cite{GK3}. Although their methods are somewhat similar to ours we feel that our approach is of independent interest.
\end{remark}

\subsection{Twisted surface diagrams of annular \swfs}
	\label{S:SWFs -> SDs}
Consider an annular \swf~$w\colon W\ra A$. We denote by $\delpm A$ the boundary components of the base annulus~$A$ contained in the higher (resp.\ lower) genus region and we let~$\delpm W=w\inv(\delpm A)$.

\begin{definition}\label{D:reference system}
Let $w\colon W\ra A$ be an annular \swf. A \emph{reference system}~$\mcR=\set{R_1,\dots,R_c}$ for~$w$ (where $c$~is the number of cusps) is a collection of reference paths for~$w$ such that
\begin{enumerate}
	\item all reference paths have the same reference points $p_\pm\in\delpm A$,
	\item the interiors of the arcs are pairwise disjoint,
	\item with respect to the orientations on~$\delpm A$ the arcs leave~$\delp A$ and enter~$\delm A$ in order of increasing index (see Figure~\ref{F:reference system}) and
	\item each fold arc is hit by exactly one of the $R_i$.
\end{enumerate}
\end{definition}

	\begin{figure}[h]
	\includegraphics{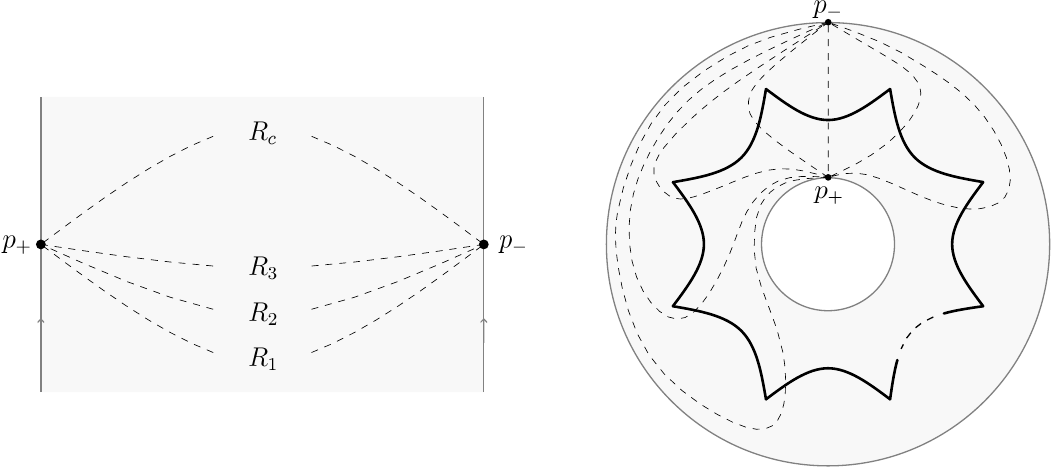}
	\caption{A reference system for an annular \swf.}
	\label{F:reference system}
	\end{figure}

As before, we denote the reference fibers by $\S_\pm:=\S_\pm(\mcR)=w\inv(p_\pm)$. Using the reference fibers we can write $\delpm W$ as mapping tori
	\begin{equation*}
	\delpm W\cong \S_\pm(\mu_\pm)
	\end{equation*}
where $\mu_\pm\in\M(\S_\pm)$ is the monodromy of $w$ over $\delpm A$ (in the positive direction). We will refer to $\mu_+$ (resp.~$\mu_-$) as the \emph{higher} (resp.\ \emph{lower}) \emph{monodromy} of~$w$.

\begin{lemma}\label{T:reference systems -> circuits}
Let $w\colon W\ra A$ be an annular \swf together with a reference system~$\mcR=\set{R_1,\dots,R_c}$ and let $\g_i=\g(R_i)\subset\S_+$. Then for $i<c$ the vanishing cycles $\g_i$~and~$\g_{i+1}$ are geometrically dual and, moreover, so are $\mu_+(\g_c)$~and~$\g_1$.
\end{lemma}

In the proof of this Lemma and subsequent consideration we will need the following notion. Let $B$ be an oriented surface and let $R\subset B$ be a proper arc which hits a boundary component $\del_i B\subset\del B$ transversely in a single point. We parametrize a small collar of~$\del_i B$ by $S^1\times[0,1]$ in such a way that~$\del_i B$ corresponds to $S^1\times\set{1}$, $R$~corresponds to $\set{1}\times[0,1]$ and the induced orientations on~$\del_i B$ agree. We say that the arc $R'$ which corresponds to 
	\begin{equation*}
	\left\{ (e^{2\pi t},t) \middle\vert t\in [0,1] \right\}
	\end{equation*}
via the parametrization is obtained from~$R$ by \emph{swinging once around~$\del_i B$}.
\begin{remark}\label{R:swinging vs. Dehn twists}
At first glance, swinging about a boundary component seems to be the same as applying a boundary parallel Dehn twist and up to isotopy this is indeed the case. However, there is a subtle difference since a boundary parallel Dehn twists is usually assumed to be supported in the interior of the surface and thus fixes a small collar of the boundary point wise while the support of the swinging diffeomorphism goes right up to the boundary. 
This difference becomes important in the following situation. 

Let~$S$ be another arc with the same properties as~$R$ such that $R$~and~$S$ are disjoint in the interior of~$B$ and assume that~$S$ leaves~$\del_i B$ after~$R$. If we swing~$S$ once around~$\del_i B$, then the resulting arc remains disjoint from~$R$ but now leaves~$\del_i B$ before~$R$. On the other hand, if we perform a boundary parallel Dehn twist on~$S$, then we keep the exit order at the price of introducing an interior intersection point.
In particular, if $\mcR=\set{R_1,\dots,R_c}$ is a reference system for an annular \swf, then we obtain a new reference system by swinging the last arc~$R_c$ once around each boundary component fo the annulus.
\end{remark}

\begin{proof}[Proof of Lemma~\ref{T:reference systems -> circuits}]
The first statement follows from Lemma~\ref{T:adjacent reference paths} since for~$i<c$ the reference paths $R_i$~and~$R_{i+1}$ are clearly adjacent. 
The second statement needs an additional arguments. We first swing~$R_c$ once around the boundary of~$A$ so that the resulting reference path~$R_c'$ is adjacent to~$R_1$ and thus $\g(R_c')$ is geometrically dual to~$\g(R_1)$. Next we observe that~$R_c'$ is homotopic to~$R_c$ precomposed with the boundary curve. Thus the parallel transport along $R_c'$ is the composition of the parallel transport along $R_c$ and the higher genus monodromy. In particular, we have $\g(R_c')=\mu_+(\g_c)$.
\end{proof}

\begin{remark}\label{R:other side of reference paths}
Note that in the above proof we did not actually need the whole reference system but only the parts of the arcs contained in the higher genus region. 
\end{remark}

Let us isolate the combinatorial structures encountered in the above Lemma.

\begin{definition}\label{D:circuits}
Let $\S$ be a surface. 
A \emph{circuit} (of \emph{length}~$c$) on~$\S$ is an ordered collection of \sccs $\Gamma=(\g_1,\dots,\g_c)$ such that any two adjacent curves $\g_i$~and~$\g_{i+1}$ are geometrically dual for~$i<c$.
A \emph{switch} for $\G$ is a mapping class $\mu\in\M(\S)$ such that $\mu(\g_c)$~and~$\g_1$ are geometrically dual. 
We say that $\G$ is \emph{closed} if $\g_c$~and~$\g_1$ are geometrically dual, i.e.~if the identity works as a switch.
\end{definition}
\begin{definition}\label{D:surface diagrams}
A \emph{twisted surface diagram} is a triple $\SD =(\S,\G,\mu)$ where $\S$ is a closed, oriented surface, $\G$~is a circuit in~$\S$ and $\mu\in\M(\S)$~is a switch for~$\G$.
\end{definition}

\begin{remark}\label{R:chains of curves}
There is no restriction on the intersections of non-adjacent curves in a circuit. Circuits in which non-adjacent curves are disjoint, so called \emph{chains of curves}, are well known objects in the theory of mapping class groups of surfaces where they play an important role.
\end{remark}
\begin{remark}\label{R:orienting circuits}
Sometimes it will be convenient to choose orientations on the curves in a circuit $\G=(\g_1,\dots,\g_c)$ in order to speak of their homology classes. If the ambient surface is oriented, we will always choose orientations such that the intersection of $\g_i$~and~$\g_{i+1}$, $i<c$, has positive sign.
\end{remark}

With this terminology we can rephrase Lemma~\ref{T:reference systems -> circuits} as stating that an annular \swf $w\colon W\ra A$ together with a reference system~$\mcR$ induces a twisted surface diagram
	\begin{equation*}
	\SD_{w,\mcR}:=(\S_+,\G_{w,\mcR},\mu_+)
	\end{equation*}
where the higher monodromy works as a switch.

Note that when the higher monodromy is trivial we obtain a closed circuit and recover Williams' surface diagrams for which we shall reserve this name, i.e.\ in the following the term surface diagram will always mean a triple $(\S,\G,\id)$ which we simply denote by $(\S,\G)$ or sometimes even $(\S;\g_1,\dots,\g_c)$. Whenever we allow nontrivial higher monodromy we will explicitly speak of twisted surface diagrams.

\medskip
Not surprisingly, the twisted surface diagrams constructed in Lemma~\ref{T:reference systems -> circuits} depend on the choice of the reference system. To understand this dependence we observe that a reference system is uniquely determined (up to isotopy relative to the boundary and the cusps) by specifying the first reference path -- this follows directly from the definition. 
Furthermore, it is easy to see that any two reference paths which have the same reference points and hit the same fold arc become isotopic after suitably swinging around the boundary components of~$A$.

Now let $\mcR=\set{R_1,\dots,R_c}$ and $\mcS=\set{S_1,\dots,S_c}$ be two reference systems with common reference points and let~$S_k$ hit the same fold arc as~$R_1$. 
As in the proof of Lemma~\ref{T:reference systems -> circuits} we successively swing the arcs $S_c,S_{c-1},\dots,S_k$ once around each boundary component to obtain a new reference system~$\mcS'$ in which the first reference path hits the same fold arc as~$R_1$. Now, by further swinging all of $\mcS'$ simultaneously, but this time independently around the boundary components, we can match the two first reference paths and thus the whole reference systems.

Let us analyze the effect of this matching procedure on the twisted surface diagram. For brevity of notation let $\SD=(\S,\G,\mu)$ be the twisted surface diagram associated to an annular \swf $w\colon W\ra A$ together with a reference system~$\RR$. Since the surface~$\S$ and the switch~$\mu$ only depend on the reference points, only the circuit~$\G=(\g_1,\dots,\g_c)$ will be affected by swinging some reference paths. Moreover, note again that the vanishing cycles~$\g_i$ only depend on the part of the reference paths contained in the higher genus region. Thus swinging around the lower genus boundary does not change the circuit. 

Now, as we have already observed, if we swing the last reference path in $\RR$ once around both boundary components, we obtain a new reference system~$\RR'$ and which induces the circuit
	\begin{equation*}
	\G_\mu^{[1]}:=\big( \mu(\g_c),\g_1,\dots,\g_{c-1} \big) .
	\end{equation*}
This operation of going from~$\SD$ to $\SD^{[1]}:=(\S,\G_\mu^{[1]},\mu)$ makes sense in the abstract setting of twisted surface diagrams and we call it (and its obvious inverse) \emph{switching}. Note that if the higher monodromy~$\mu$ is trivial, then switching amounts to a cyclic permutation of the vanishing cycles.

Since we can relate any two reference systems for a given annular \swf by suitably swinging reference paths, we see that the twisted surface diagram is well defined up to switching.

\medskip
Next we want to compare the twisted surface diagrams of two equivalent annular \swfs as in the commutative diagram below.
	\begin{equation*}
	\xymatrix{X\ar[d]_w \ar[r]^{\hat{\phi}} & X'\ar[d]^{w'} \\ A\ar[r]^{\check{\phi}} & A'}
	\end{equation*}
If $\RR$ is a reference system for $w$, then $\RR':=\check{\phi}(\RR)$ is a reference system for~$w'$. Let~$\SD=(\S,\G,\mu)$  and~$\SD'=(\S',\G',\mu')$ be the associated twisted surface diagrams. Then $\hat{\phi}$ induces an orientation preserving diffeomorphism $\phi\colon\S\ra \S'$ and clearly the higher monodromies satisfy~$\mu'=\phi\mu\phi\inv$. It is also easy to see that 
	\begin{equation*}
	\G'=\phi(\G):=\big( \phi(\g_1),\dots,\phi(\g_c) \big)
	\end{equation*}
where, as usual, $\G=(\g_1,\dots,\g_c)$. 
Again, the effect of an equivalence of annular \swfs makes sense for abstract twisted surface diagrams and we say that $\SD$~and~$\SD'$ are \emph{diffeomorhpic} via~$\phi$. Putting this together with switching we end up with the following definition.

\begin{definition}\label{D:equivalence of SDs}
Two twisted surface diagrams $\SD$ and $\SD'$ called \emph{equivalent} if, for some integer~$k$, $\SD'$ is diffeomorphic to~$\SD^{[k]}$.
\end{definition}

Summing up the content of this section we have proved the first half of Theorem~\ref{T:annular SWFs <-> twisted SDs}:
\begin{proposition}\label{T:annular SWFs -> twisted SDs}
To an annular \swf $w\colon W\ra A$ we can assign a twisted surface diagram 
	\begin{equation*}
	\SD_w=(\S_+,\G_w,\mu_+)
	\end{equation*}
which is well defined up to switching. Moreover, equivalent annular \swfs have equivalent twisted surface diagram.
\end{proposition}

\begin{remark}\label{R:picturing surface diagrams}
We would like to point out that it is very convenient that only the equivalence class of the surface diagram plays a role. Indeed, in order to actually visualize the twisted surface diagram of an annular \swf one has to identify the higher genus reference fiber with some model surface and there is no canonical way to do so. However, any two such identifications will differ by a diffeomorphism of the model surface and thus be equivalent. So we can safely forget about the choice of identification whenever we are only interested in the equivalence class of the \swfs or the diffeomorphism type of its total space.
\end{remark}

\subsection{Handle decompositions for annular \swfs}
	\label{S:handle decompositions}
As a next step we relate the twisted surface diagrams associated to annular \swfs to the topology of their total spaces. We will see that the situation is very similar to Lefschetz fibrations

\begin{proposition}\label{T:handle decompositions from SWFs}
Let $w\colon W\ra A$ be an annular \swf. Then $W$ has a relative handle decomposition on $\delp W$ with one 2-handle for each fold~arc. Such a handle decomposition is encoded in any twisted surface diagram for~$w$.
\end{proposition}
In the following we will refer to the 2-handles associated to the fold arcs as \emph{fold handles}. 
\begin{proof}
The rough idea is to parametrize~$A$ by the model annulus~$S^1\times [0,1]$ such that the composition of~$w$ and the projection $p\colon S^1\times [0,1]\ra [0,1]$ becomes a Morse function. The details go as follows.

We equip $S^1\times [0,1]$ with coordinates $(\theta,t)$ refer to the direction in which~$t$ increases as \emph{right}.
We say that a parametrization $\kappa\colon A\ra S^1\times [0,1]$ is \emph{$w$-regular} if the critical image $C_\kappa:=\kappa\circ w(\Crit_w)$ is in the following \emph{standard position}:
\begin{itemize}
	\item all cusps point to the right
	\item each $R_\theta:=\{\theta\}\times[0,1]$ meets $C_\kappa$ in exactly one point, either in a cusp or transversely in a fold point and
	\item the projection $p$ restricted to $C_\kappa$ has exactly one minimum on each fold arc.
\end{itemize}
We claim that for any $w$-regular parametrization $\kappa$, the map
	\begin{equation*}
	p_\kappa:=p\circ\kappa\circ w\colon W\ra [0,1]
	\end{equation*}
is a Morse function.
Clearly, the critical points of $p_\kappa$ are contained in~$\Crit_w$. Thus we have to understand how the projection~$p$ interacts with the critical image~$C_\kappa$. By the standard position assumption there are three ways how a level set $S_t:=S^1\times\set{t}$ can intersect~$C_\kappa$ (see Figure~\ref{F:slices}):
\begin{enumerate}
	\item[a)] $S_t$ intersects $C_\kappa$ transversely in a fold point,
	\item[b)] $S_t$ meets $C_\kappa$ in a cusp and the fold arcs surrounding the cusp are on the left side of~$S_t$ or
	\item[c)] $S_t$ is tangent to a fold arc which is located on the right side of~$S_t$. We will refer to this phenomenon as a \emph{concave tangency}.
\end{enumerate}
	\begin{figure}[h]
	\includegraphics{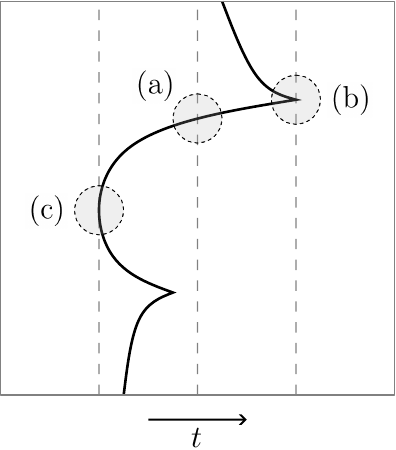}
	\caption{Level sets intersecting the critical image.}
	\label{F:slices}
	\end{figure}
It turns out that only the concave tangencies contribute critical points of~$p_\kappa$. In fact, from the models for the fold an cusp we immediately see that~$p_\kappa$ is modeled on the compositions
	\begin{equation}
	(t,x,y,z)\mapsto(t,-x^3+3tx-y^2+z^2)\mapsto t
	\end{equation}
in case of a cusp intersection and
	\begin{equation}
	(t,x,y,z)\mapsto(t,-x^2-y^2+z^2)\mapsto \pm t
	\end{equation}
for a transverse fold intersection\footnote{The sign depends on how the fold and cusp models are embedded.} which shows that these are regular points of~$p_\kappa$. 

It remains to treat the concave tangencies. These occur precisely at the minima of~$p_\kappa|_{C_\kappa}$. This minimum can be modeled by $t\mapsto t^2$ and 
it is easy to see that
$p_\kappa$ is modeled on 
	\begin{equation}
	(t,x,y,z)\mapsto(-x^2-y^2+z^2+t^2)
	\end{equation}
which is a Morse singularity of index~2. By assumption there is exactly one concave tangency for each fold arc and, using the correspondence between Morse functions and handle decompositions, we obtain the desired handle decomposition.

In order to understand how the fold handles are attached consider the arcs $R_i:=R_{\theta_i}\subset S^1\times [0,1]$ where $\theta_1,\dots,\theta_c\in S^1$ is a sequence of numbers ordered according to the orientation of~$S^1$ (e.g.~the $c$-th~roots of unity). 
The $w$-regular parametrization $\kappa$ can be chosen in such a way that each~$R_i$ is a reference path for precisely one fold arc and~$C_\kappa$ is contained in the open annulus~$S^1\times (\epsilon,1-\epsilon)$ for some~$\epsilon>0$. 
For each $R_i$ we obtain a vanishing cycle~$\g_i$ in the fiber of~$w$ over~$(\theta_i,0)\in\delp A$ and the local model for the fold singularity implies that the fold handles are attached to $\delp W\times[0,\epsilon]$ along the vanishing cycles~$\g_i$ pushed off into the fiber over~$(\theta_i,\epsilon)$ with respect to the canonical framing induced by the fiber.

The relation to twisted surface diagrams now becomes obvious. There is a canonical way to turn the reference paths $\Theta_1,\dots,\Theta_c$ into a reference system by fixing $\Theta_1$ and successively sliding the endpoints of the remaining arcs along the boundary onto~$\Theta_1$ against the orientation. Thus the vanishing cycles record the attaching curves of the fold handles.
\end{proof}

\begin{remark}\label{R:why we need cusps}
The above proposition is one of the reasons that made us require the presence of cusps in the critical loci of \swfs. If there were no cusps, then it would not be possible to avoid \emph{convex tangencies} which correspond to 3-handles instead of 2-handles. Thus the presence of cusps guarantees that the total spaces of annular \swfs are (relative) 2-handlebodies.
\end{remark}
\begin{remark}\label{R:vertical tangencies}
The observation that fold tangencies correspond to Morse singularities also appears in~\cite{GK1} in their more general setting of Morse 2-functions.
The fact that the real part of the Lefschetz model is also a Morse function allows to include Lefschetz singularities in the discussion. Proceeding this way, one can recover Baykur's result about handle decompositions from broken Lefschetz fibrations (see~\cite{Baykur2}).
\end{remark}
\begin{remark}\label{R:similarity with Lefschetz fibrations}
The reader familiar with Lefschetz fibrations will have noticed the strong resemblance of the handle decompositions described above with the ones induces by Lefschetz fibrations. In fact, the handle decompositions have exactly the same structure except that the fold handles are attached with respect to the fiber framing while the framing of the \emph{Lefschetz handles} differs by~$-1$.
\end{remark}

\subsection{Annular \swfs from twisted surface diagrams}
	\label{S:SDs -> SWFs}
Using the handle decompositions exhibited in the previous section as a stepping stone we can now build annular \swfs out of twisted surface diagrams and thus complete the proof of Theorem~\ref{T:annular SWFs <-> twisted SDs}.

\begin{proposition}\label{T:annular SWFs <- twisted SDs}
A twisted surface diagram $\SD=(\S,\G,\mu)$ determines an annular \swf $w_\SD\colon W_\SD\ra S^1\times[0,1]$ with higher genus fiber~$\S$ and higher monodromy~$\mu$. 
\end{proposition}
\begin{proof}
To make the construction of $w_\SD$ more transparent we begin with some preliminary considerations.

One important ingredient is the mapping cylinder $\S(\mu)$ which is equipped with a canonical fibration $p\colon\S(\mu)\ra S^1$. 
Given the construction of~$\S(\mu)$ it is convenient to consider~$S^1$ as the quotient~$[0,1]/\{0,1\}$ and we will identify~$\S$ with the fiber over the point~$0\sim1$.

We will now describe a collection of arcs~$\RR=\{R_1,\dots,R_c\}$ in $S^1\times[0,1]$, which we consider as
	\begin{equation*}
	S^1\times[0,1]=[0,1]\times[0,1] / (0,t)\sim(1,t),
	\end{equation*}
that will serve as a reference system for~$w_\SD$ (see Figure~\ref{F:builsing SWFs}~(a)).

Let $r\colon[0,1]\ra [0,1]$ be a smooth function that has the constant value~$1$ on the interval $[\tfrac{1}{3},\tfrac{2}{3}]$, satisfies $r(0)=r(1)=0$ and is strictly increasing (resp.\ decreasing) for $t\leq\tfrac{1}{3}$ (resp.\ $t\geq\tfrac{2}{3}$).
If the length of~$\G$ is $c$, then for $i=1,\dots,c$ we let $\theta_i:=\tfrac{i-1}{c}$ and define
	\begin{equation*}
	R_i:= 
	\big\{ 
		\big(\theta_i r(t),t\big)/\sim
		\;\big\vert \;
		t\in[0,1] \subset S^1\times[0,1] 
	\big\}
	\end{equation*}

	\begin{figure}
	\includegraphics{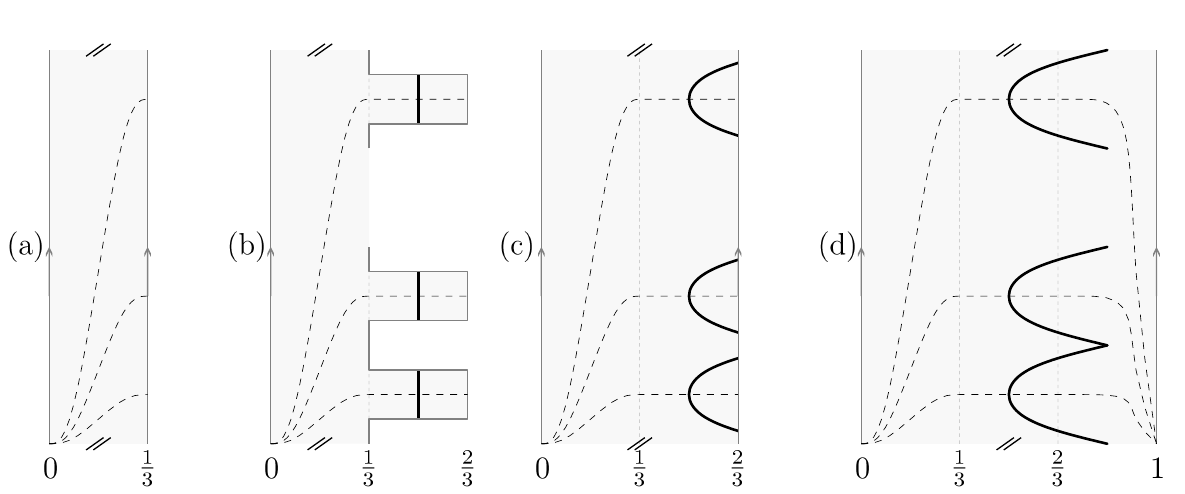}
	\caption{Building a \swf from a surface diagram. (bold: critical image, dashed: reference path)}
	\label{F:builsing SWFs}
	\end{figure}

\medskip
We these remarks in place we can now begin with the construction of~$W_\SD$ and~$w_\SD$. This will be done in three steps.

\smallskip
\textit{Step 1:}
We begin by taking the product 
	\begin{equation*}
	W_1:=\S(\mu)\times[0,\tfrac{1}{3}]
	\end{equation*}
and define a map $w_1\colon W_1\ra S^1\times[0,\tfrac{1}{3}]$ by sending $(x,t)$ to $(p(x),t)$.

\smallskip
\textit{Step 2:} 
Next, we construct~$W_2$ by attaching 2-handles to~$W_1$ in the following way. Let $\G=(\g_1,\dots,\g_c)$. 
Using the arc~$R_i\subset S^1\times[0,1]$ described above we can parallel transport the curve~$\g_i\subset\S$ to the fiber of~$w_1$ over~$(\theta_i,\tfrac{1}{3})$. We attach a 2-handle to the resulting curve with respect to the fiber framing. 

This choice of framing allows us to extend $w_1$ over each 2-handle. Indeed, we can consider attaching the $i$-th~(4-dimensional) 2-handle as a 1-parameter family of 3-dimensional 2-handle attachments parametrized by a small \nbhd of~$(\theta_i,1)$ in $S^1\times\set{1}$. (Of course, these \nbhds are pairwise disjoint.) 
For each point~$\theta$ in such a \nbhd, the restriction of~$w_1$ to the \emph{$\theta$-ray}~$\set{\theta}\times[0,\tfrac{1}{3}]$ extends to a Morse function (with one critical point of index 2) over a slightly longer ray, say $\set{\theta}\times[0,\tfrac{2}{3}]$, in the standard way. 
Using these 1-parameter families of Morse functions we can extend~$w_1$ to map from~$W_2$ to an annulus with ``bumps'' on one side as shown in Figure~\ref{F:builsing SWFs}~(b) and this map has an arc of indefinite folds on each bump. 
We can then smooth out the bumps by standard techniques from differential topology to obtain a map $w_2\colon W_2\ra S^1\times[0,\tfrac{2}{3}]$ in which each 2-handle attachment has created an arc of indefinite folds whose endpoints hit the boundary of~$W_2$ transversely in the component that was affected by the handle attachment (Figure~\ref{F:builsing SWFs}~(c)), let us call this component~$\del_2 W_2$

\smallskip
\textit{Step 3:} 
For the final step we first note that the restriction of $w_2$ over $S^1\times\{\tfrac{2}{3}\}$ is a circle valued Morse function with a pair of critical points of index~1 and~2 for each fold arc of~$w_2$. 
The crucial observation is that the condition that~$\G$ is a circuit with switch~$\mu$ implies that all these pairs of critical points cancel! 
Thus there is a standard homotopy, which we parametrize by $[\tfrac{2}{3},1]$, from $w_2|_{\del_{2}W_2}$~to a submersion that realizes this cancellation.
We let
	\begin{equation*}
	W_\SD:=W_2\cup_{\del_{2}W_2} \del_{2}W_2\times[\tfrac{2}{3},1]
	\end{equation*}
and extend $w_2$ by tracing out the homotopy over the newly added collar of~$\del_{2}W_2$ to obtain a map~$w_\SD\colon W_\SD\ra S^1\times[0,1]$. This last step removes all critical points from the boundary and introduces an interior cusp for any canceling pair. Clearly $w_\SD$ is an annular \swf with base diagram as in Figure~\ref{F:builsing SWFs}~(d).

Note that~$W_\SD$ is diffeomorphic to~$W_2$ and thus has the same relative handle decomposition.
Moreover, it follows directly from the construction that~$\RR$ is a reference system for~$w_\SD$ with~$\SD$ as its twisted surface diagram.
\end{proof}
In order to finish the proof of Theorem~\ref{T:annular SWFs <-> twisted SDs} we have to show that equivalent twisted surface diagram give equivalent annular \swfs. Recall that an equivalence of surface diagram is a combination of two things: switching and a diffeomorphism. We will treat these separately.

\begin{lemma}\label{T:diffeomorphis SDs -> equivalent SWFs}
If $\SD$ and $\SD'$ are diffeomorphic, then $w_\SD$ and $w_{\SD'}$ are equivalent.
\end{lemma}
\begin{proof}
Let $\SD=(\S,\G,\mu)$, $\SD'=(\S',\G',\mu')$ and let $\phi\colon\S\ra \S'$ be a diffeomorphism such that~$\G'=\phi(\G)$ and~$\mu'=\phi\mu\phi\inv$. We will extend~$\phi$ to a diffeomorphism $\hat{\phi}\colon W_\SD\ra W_{\SD'}$ which fits in the commutative diagram
	\begin{equation*}
	\xymatrix{
		W_\SD\ar[dr]_{w_\SD} \ar[rr]^{\hat{\phi}} & & W_{\SD'}\ar[dl]^{w_{\SD'}} \\
		& S^1\times[0,1] &
	}
	\end{equation*}
This will be done by going through the steps in the proof of Proposition~\ref{T:annular SWFs <- twisted SDs}. Let $W_i$ and $W_i'$, $i=1,2$, denote the 4-manifolds built in each step.

From the identity~$\mu'=\phi\mu\phi\inv$ we see that~$\phi$ induces a fiber preserving diffeomorphism~$\S(\mu)\ra \S'(\mu')$. Taking the product with the identity, we obtain~$\hat{\phi}_1\colon W_1\ra W_1'$.

In the second step, where the 2-handles are attached to the curves in~$\G$, we simply note that $\hat{\phi}_1$ maps the attaching regions into each other and can thus be extended over the 2-handles to $\hat{\phi}_2\colon W_2\ra W_2'$. Note that the smoothing of the bumpy annulus does not cause any trouble since it does not involve the total space.

For the third step observe that, given a homotopy from $w_2|_{\del_2 W_2}$ to a submersion, we can push it forward via $\hat{\phi}_2|_{\del_2 W_2}$ to obtain such a homotopy for $w_2'|_{\del_2 W_2'}$.
\end{proof}

\begin{lemma}\label{T:switching -> equivalent SWFs}
If $\SD$ is a twisted surface diagram, then $w_\SD$ and $w_{\SD^{[1]}}$ are equivalent.
\end{lemma}
\begin{proof}
If we take the canonical reference system for~$w_\SD$ and swing the last reference path once around the boundary, we obtain a reference system which induces~$\SD^{[1]}$. Thus~$w_\SD$ and~$w_{\SD^{[1]}}$ can be considered as the same annular \swf.
\end{proof}

Combining these two lemmas we obtain
\begin{corollary}\label{T:equ SDs -> equ SWFs}
If $\SD$ and $\SD'$ are equivalent, then so are $w_\SD$ and $w_{\SD'}$.
\end{corollary}

\subsection{Gluing ambiguities}
	\label{S:Gluing ambiguities}
Now that we know how to study annular \swfs in terms of their twisted surface diagrams, recall that \swfs over arbitrary base surfaces can be obtained from an annular ones by gluing suitable surface bundles to the boundary components.
To be precise, let $w_0\colon W\ra A$ be an annular \swf and let $\pi_\pm\colon Y_\pm\ra B_\pm$ be surface bundles over surfaces~$B_\pm$ such that there are boundary components $C_\pm\subset B_\pm$ and fiber preserving diffeomorphisms $\psi_\pm\colon \pi_\pm\inv(C_\pm)\ra \delpm W$. Then we can form a \swf
	\begin{equation*}
	w\colon Y_+\cup_{\psi_+} W \cup_{\psi_-} Y_- \lra B_+ \cup_{C_+} A \cup_{C_-} B_-.
	\end{equation*}
Of course, different choices of gluing diffeomorphisms may lead to inequivalent \swfs. 
If we fix a pair~$\psi_\pm$ of gluing maps, then we can obtain any other such pair by composing with automorphisms (in the sense of Section~\ref{S:mapping tori}) of the boundary fibrations $w_0\colon\delpm W\ra S^1$. 
Obviously, isotopic gluing maps give rise to equivalent \swfs and the gluing ambiguities are a priori parametrized by 
	\begin{equation*}
	\pi_0\big(\Aut(\delp W, w)\big) \times \pi_0\big(\Aut(\delm W,w)\big)
	\end{equation*}
However, it turns out that the first factor can be eliminated.

\begin{lemma}\label{T:boundary diffeomorphisms extend}
Let $w\colon W\ra A$ be an annular \swf. Then any fiber preserving diffeomorphism of~$\delp W$ extends to an auto-equivalence of~$w$.
\end{lemma}
\begin{proof}
By Theorem~\ref{T:annular SWFs <-> twisted SDs} we can assume that~$w$ is built from a twisted surface diagram~$\SD=(\S,\G,\mu)$ such that~$\delp W=\S(\mu)$.
According to~\eqref{E:automorphisms of mapping tori} there are two types of automorphisms of~$\S(\mu)$, the \emph{constant} ones coming from~$C_{\M(\S)}(\mu)$ and the \emph{non-constant} ones originating from~$\pi_1(\Diff(\S),\id)$.
The statement that constant automorphisms of~$\delp W$ extend to auto-equivalences of~$w$ is just a reformulation of Lemma~\ref{T:diffeomorphis SDs -> equivalent SWFs}. Thus it remains to treat the non-constant ones.

Recall that by Theorem~\ref{T:Earle-Eells} these only occur when~$\S$ has genus one. We can thus assume that $\S=T^2$. 
A well known refinement of Theorem~\ref{T:Earle-Eells} states that the map
	\begin{equation}\label{E:torus iso}
	\pi_1\big(\Diff(T^2),\id\big)\ra \pi_1(T^2,x)
	\end{equation}
which sends an isotopy to the path traced out by a base point~$x\in T^2$ during that isotopy is an isomorphism (see~\cite{Earle-Eells}).
Note that the fundamental group of~$T^2$ is generated by the curves~$\g_1$ and~$\g_2$ (after choosing orientations, of course) if we take their unique intersection point as base point. Hence, we only have to extend the automorphisms coming from generators of~$\pi_1\big(\Diff(T^2),\id\big)$ mapping to~$\g_1$ and~$\g_2$ in~\eqref{E:torus iso}.
If one parametrizes the torus by $S^1\times S^1\subset \C^2$ such that $S^1\times\set{1}$ maps to~$\g_1$ and $\set{1}\times S^1$ maps to~$\g_2$, then such generators are given by
	\begin{equation*}
	h^{\g_1}_t(\xi,\eta):=\big( e^{2\pi i\,t}\xi, \eta \big)
	\quad\text{and}\quad
	h^{\g_2}_t(\xi,\eta):=\big( \xi, e^{2\pi i\,t}\eta \big)
	\quad\quad
	(t\in[0,1])
	\end{equation*}
and we denote the corresponding automorphisms of~$\S(\mu)$ by 
	\begin{equation*}
	\varphi_i(x,t):=\big( h^{\g_i}_t(x),t \big).
	\end{equation*}
In order to extend~$\varphi_i$ to~$Z_\SD$ we take one step back and homotope the path~$h^{\g_i}$ to be constant outside the interval where the 2-handle corresponding to~$\g_i$ is attached. These intervals (times [0,1]) are highlighted in Figure~\ref{F:gluing}.
	\begin{figure}
	\includegraphics{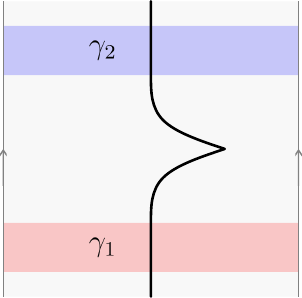}
	\caption{The relevant regions for extending non-constant automorphisms.}
	\label{F:gluing}
	\end{figure}
Outside the preimage of the regions shown in Figure~\ref{F:gluing} we can simply extend~$\varphi_i$ as the identity. In these region, observe that~$h^{\g_i}_t$ fixes~$\g_i$ set wise at all times, it just rotates it more and more as~$t$ increases.
It is easy to see that these rotations can be extended across the 2-handles in a way that respects the fibration structure. 
\end{proof}

\begin{remark}\label{R:gluing without cusps}
The genus one case of Example~\ref{eg:ADK sphere} shows that this Lemma does not hold without in the absence of cusps. The above proof breaks down at the point where we need the vanishing cycles to generate the fundamental group.
\end{remark}

\section{\Swfs over the disk and the sphere}
	\label{S:SWFs over disk and sphere}
We now leave the general theory behind and focus on untwisted surface diagrams, i.e.\ pairs~$(\S,\G)$ where~$\G$ is a closed circuit in~$\S$. By the correspondence established in the previous section such a surface diagram corresponds to an annular \swf whose higher genus boundary component has trivial monodromy. We can thus fill this boundary component with~$\S\times D^2$ using some fiber preserving diffeomorphism of~$\S\times S^1$ to obtain a \swf over the disk. 
(Note that the boundary of the disk is contained in the lower genus region; we will refer to such fibrations as \emph{descending} \swfs over the disk.) 
Furthermore, by Lemma~\ref{T:boundary diffeomorphisms extend} different choices of gluing diffeomorphisms produce equivalent \swfs. 
Altogether, we have established the following.
\begin{proposition}\label{T:SDs and descending SWFs}
There is bijective correspondence between (untwisted) surface diagrams up to equivalence and descending \swfs over the disk up to equivalence.
\end{proposition}
As mentioned before, when we speak of surface diagrams, we will always mean untwisted surface diagrams. This will not lead to confusion since we will not encounter any twisted surface diagrams anymore.

For a surface diagram $\SD=(\S,\G)$ we denote the corresponding \swf by $w_\SD\colon Z_\SD\ra D^2$ or, by a slight abuse of notation, simply by~$Z_\SD$ with the map to the disk implicitly understood. The boundary of~$Z_\SD$ fibers over~$S^1$ and if this boundary fibration is trivial, then we can \emph{close off} to a \swf over~$S^2$. Recall that Theorem~\ref{T:Williams existence} tells us that we can obtain \emph{all} smooth, closed, oriented 4-manifolds by this process. It is thus of great interest to understand which surface diagrams describe closed 4-manifolds. The following example indicates that this might be a hard problem.

\begin{example}\label{eg:arbitrary monodromy}
Let $\S$ be a closed, orientable surface together with a mapping class $\phi\in\M(\S)$. Then any factorization of $\mu$ into positive Dehn twists yields a Lefschetz fibration over the disk whose boundary can be identified with the mapping torus $\S(\phi)=(\S\times[0,1])/(x,1)\sim(\phi(x),0)$. As in Example~\ref{eg:wrinkled Lefschetz fibrations} we can turn this Lefschetz fibration into a descending \swf without changing the boundary. Thus any surface bundle over the circle (with closed fibers) bounds some descending \swf over the disk and any mapping class can be realized as the monodromy of a surface diagram.
\end{example}

In fact, the situation is very similar to the theory of Lefschetz fibrations. 
Any word in positive Dehn twists (or, equivalently, a finite sequence of \sccs) on a closed, oriented surface determines a Lefschetz fibration over the disk, the boundary fibers over the circle with monodromy given by the product of the Dehn twists and if this monodromy is trivial, then one can close off to a Lefschetz fibration over~$S^2$. 
Just as an arbitrary product of Dehn twists will will not be isotopic to the identity, a surface diagram will not give rise to a \swf over~$S^2$.
The advantage of the Lefschetz setting is the direct control over the boundary.

\subsection{The monodromy of a surface diagram}
	\label{S:monodromy of SDs}
In order to obtain a more intrinsic description of the boundary of~$Z_\SD$ in terms of $\SD$ we need a little detour. 

Let $a,b\subset \S$ be a pair of \sccs in a surface~$\S$ that intersect transversely in a single point. 
We denote by~$\S_a$ and~$\S_b$ the surfaces obtained by surgery on the curves~$a$ and~$b$, respectively. To be concrete, we fix tubular \nbhds $\nu a$ and~$\nu b$ and picture~$\S_a$ (resp.~$\S_b$) as the result of filling in the two boundary components of~$\S\setminus\nu a$ (resp.~$\S\setminus\nu b$) with disks.
By the assumption on intersections we can assume that~$\nu(a\cup b):=\nu a\cup\nu b$ is diffeomorphic to a once punctured torus -- for convenience we will also assume that it has a smooth boundary in~$\S$. 
Observe that~$\S\setminus \nu(a\cup b)$ has one boundary component and is contained in both $\S_a$ and~$\S_b$ as a subsurface.
Furthermore, the closure of~$\nu b\setminus\nu a$ (resp.~$\nu a\setminus\nu b$) is a disk in~$\S_a$ (resp.~$\S_b$).
It follows that, up to isotopy, there is a unique diffeomorphism
	\begin{equation*}
	\kappa_{a,b}\colon\S_a\ra \S_b
	\end{equation*}
which can be assumed to map~$\nu b\setminus \nu a$ onto $\nu a\setminus\nu b$.

\smallskip
Now let $\SD=(\S;\g_1,\dots,\g_l)$ be a surface diagram and consider the associated \swf $w_\SD\colon Z_\SD\ra D^2$.
Then each adjacent pair of curves~$\g_i$ and~$\g_{i+1}$ fits the above situation and we thus get a collection of diffeomorphisms
	\begin{equation*}
	\kappa_{\g_i,\g_{i+1}}\colon \S_{\g_i}\ra \S_{\g_{i+1}}.
	\end{equation*}
Moreover, it follows from the definition of surface diagrams that the composition
	\begin{equation*}
	\mu_\SD
	:=\kappa_{\g_c,\g_1}\circ\kappa_{\g_{c-1},\g_c}\circ \dots\circ\kappa_{\g_1,\g_2}
	\end{equation*}
maps~$\S_{\g_1}$ to itself and it is easy to see that its isotopy class does not depend on any of the implicit choices involved in its definition.

\begin{definition}\label{D:monodromy of SDs}
The mapping class~$\mu_\SD\in\M(\S_{\g_1})$ represented by the diffeomorphism above is called the \emph{monodromy} of~$\SD$.
\end{definition}
This name is justified by the following lemma.

\begin{lemma}\label{T:boundary monodromy}
Let $\SD=(\S,\G)$ be a surface diagram. Then 
the boundary fibration~$(\del Z_\SD,w_\SD)$ can be identified with the mapping torus~$\S_{\g_1}(\mu_\SD)$.
\end{lemma}
\begin{proof}
By the construction of~$w_\SD$ its fiber over the origin is naturally identified with~$\S$.
Furthermore, recall that the annular fibration associated to~$\SD$ is equipped with a reference system whose reference paths we can naturally extend from the annulus to the disk by connecting them to the origin. The result is a collection of reference paths~$R_1,\dots,R_c$ from the origin to the boundary of the disk and we denote its endpoints by~$\theta_1\dots,\theta_c\in S^1$.
Observe that such a reference path, $R_i$~say, gives rise to an identification of the fiber over~$\theta_i$ with the surface~$\S_{\g_i}$ obtained from surgery on~$\g_i$ where~$\g_i$ is the vanishing cycle associated to~$R_i$.

Now consider the region in the base bounded by two adjacent reference path~$R_i$ and~$R_{i+1}$. Using a suitable notion of parallel transport we see that the preimage of this region contains a trivial bundle with fiber~$\S\setminus{\nu(\g_i\cup\g_{i+1})}$. 
In particular, the parallel transport along the boundary segment from~$\theta_i$ to~$\theta_{i+1}$ restricts to the identity on the complement of~$\nu(\g_i\cup\g_{i+1})$ and thus must be isotopic to $\kappa_{\g_i,\g_{i+1}}$ and the claim follows.
\end{proof}

It is also possible to describe the monodromy in terms of the original surface~$\S$.
This takes us on another small detour.
Let $a\subset\S$ be a non-separating \scc in a surface~$\S$ and let~$\M(\S,a)$ denote the subgroup of~$\M(S)$ consisting of all elements that fix~$a$ up to isotopy. 
It is well known that there is a short exact sequence\footnote{For a proof that~$\mathrm{cut}_a$ is well defined see~\cite[Section~7.5]{Ivanov}, the rest follows as in~\cite[Chapter~3]{primer}.}
	\begin{equation}\label{E:cut sequence}
	\xymatrix{
	1 \ar[r] & 
	\scp{\tau_a} \ar[r] & 
	\M(\S,a) \ar[r]^{\mathrm{cut}_a} & 
	\M(\S\setminus a) \ar[r] & 
	1
	}
	\end{equation}
where~$\S\setminus a$ is viewed as a twice punctured surface.
The complement~$\S\setminus a$ can be related to the surgered surface~$\S_a$ as follows.
In~$\S_a$ there is an obvious pair of points, namely the centers of the surgery disks. If we denote by~$\S_a^*$ the surface obtained by marking these points, then~$\S\setminus a$ is canonically identified (at least up to isotopy) with~$\S_a^*$ and thus~$\M(\S\setminus a)$ is canonically isomorphic to~$\M(\S_a^*)$.
Hence, we can define the \emph{surgery homomorphism}
	\begin{equation*}
	\sigma_a\colon \M(\S,a)\ra \M(\S_a)
	\end{equation*}
as the composition
	\begin{equation*}
	\xymatrix{
	\M(\S,a) \ar@/^1pc/[rrr]^{\sigma_a} \ar[r]_{\mathrm{cut}_a} & 
	\M(\S\setminus a) \ar[r]_{\cong} &
	\M(\S_a^*) \ar[r]_{\mathrm{forget}} &
	\M(\S_a)
	}
	\end{equation*}
where the last map is induced by forgetting the marked points in~$\S_a^*$.

Applying this to surface diagram we obtain the following.
\begin{lemma}\label{T:lifting the monodromy}
Let $\SD=(\S;\g_1,\dots,\g_c)$ be a surface diagram. Then 
	\begin{equation*}
	\tilde{\mu}_\SD:=\tau_{\tau_{\g_c}(\g_1)} \circ
		\tau_{\tau_{\g_{c-1}}(\g_c)} \circ
		\tau_{\tau_{\g_1}(\g_2)} \in\M(\S)
	\end{equation*}
is contained in $\M(\S,\g_1)$ and satisfies $\sigma_{\g_1}(\tilde{\mu}_\SD)=\mu_\SD$.
\end{lemma}
\begin{proof}
We claim that this follows from the observation that 
	\begin{equation*}
	\tau_{\tau_{\g_i}(\g_{i+1})}(\g_i)=\tau_{\g_{i}}\tau_{\g_{i+1}}\tau_{\g_{i}}\inv(\g_i)=\g_{i+1}.
	\end{equation*}
Indeed, this obviously implies the first statement and the second follows from the fact that the diagrams
	\begin{equation*}
	\xymatrix{
	\S \ar[d]^{\tau_{\tau_{\g_i}(\g_{i+1})}} &
	\S\setminus\g_{i} \ar[d] \ar[l] \ar[r] &
	\S_{\g_{i}}^* \ar[d]^{\kappa_{\g_i,\g_{i+1}}}
	\\
	\S &
	\S\setminus\g_{i+1} \ar[l] \ar[r] &
	\S_{\g_{i+1}}^*
	}
	\end{equation*}
commute up to isotopy.
\end{proof}

The above makes it interesting to study the map~$\sigma_{\g_1}$ and its kernel. 
\begin{lemma}\label{T:generating MCG ficing a curve}
Let $a\subset\S$ be a non-separating \scc. Then the group~$\M(\S,a)$ is generated by elements of the form~$\tau_c$ where~$i(a,c)=0$ and $\Delta_{a,b}:=(\tau_a\tau_b)^3$ where~$i(a,b)=1$.
\end{lemma}
We will refer to the mapping classes~$\Delta_{a,b}$ as \emph{$\Delta$-twists}. 
\begin{proof}
It follows from the short exact sequence~\eqref{E:cut sequence} that we can obtain a generating set for~$\M(\S,a)$ by lifting a generating set for~$\M(\S\setminus a)$ and adding the Dehn twist about~$a$.
As a generating set for~$\M(\S\setminus a)$ we can take the collection Dehn twists and so called \emph{half-twists} about simple arcs connecting the two punctures.
Then the Dehn twists in~$\M(\S\setminus a)$ have obvious lifts in~$\M(\S)$ and it is easy to see that each half-twist lifts to a $\Delta$-twist.
\end{proof}

\begin{corollary}\label{T:kernel of surgery homomorphism}
The kernel of the surgery homomorphism~$\sigma_a\colon \M(S,a)\ra \M(\S_a)$ contains the Dehn twist about~$a$ and all $\Delta$-twists involving~$a$.
\end{corollary}

The expert will have noticed that the mapping class~$\tilde{\mu}_\SD$ in Lemma~\ref{T:lifting the monodromy} is simply the monodromy of the boundary of the Lefschetz part of the simplified broken Lefschetz fibration obtained from~$w_\SD$ by unsinking all the cusps.
Of course, there are many different lifts of~$\mu_\SD$ to~$\M(\S)$. 
For example, it follows from the braid relations for the pairs of adjacent curves that
	\begin{align*}
	\tilde{\mu}_\SD
	&= \tau_{\g_1}^{-c} 
		(\tau_{\g_c} \tau_{\g_1})
		(\tau_{\g_{c-1}} \tau_{\g_c})
		\dots 
		(\tau_{\g_1} \tau_{\g_2}) \\
	&= \tau_{\g_1}^{-2c} 
		(\tau_{\g_c} \tau_{\g_1} \tau_{\g_c}) 
		(\tau_{\g_{c-1}} \tau_{\g_c} \tau_{\g_{c-1}})
		\dots 
		(\tau_{\g_1} \tau_{\g_2} \tau_{\g_1})
	\end{align*}
and since~$\tau_{\g_1}$ is contained in the kernel of~$\sigma_{\g_1}$ we obtain two other choices.

\smallskip
We illustrate these mapping class group techniques to produce many examples of surface diagrams with trivial monodromy.
\begin{example}\label{eg:double monodromy}
Given an arbitrary circuit $\G=(\g_1,\dots,\g_l)$ in an oriented surface~$\S$ we can form a closed circuit $D\G:=(\g_1,\dots,\g_{l-1},\g_l,\g_{l-1},\dots,\g_2)$ which we call the \emph{double} of~$\G$.
We claim that the surface diagram~$D\SD:=(\S,D\G)$ has trivial monodromy.
For convenience let us write $\tau_i=\tau_{\g_i}$. 
As explained above the monodromy of~$D\SD$ can be lifted to~$\M(\S)$ as
	\begin{align*}
	\mu&=
		(\tau_2 \tau_1 \tau_2)
		\dots
		(\tau_{l-2} \tau_{l-1} \tau_{l-2})
		(\tau_{l-1} \tau_{l} \tau_{l-1})
		(\tau_{l} \tau_{l-1} \tau_{l})
		(\tau_{l-1} \tau_{l-2} \tau_{l-1})
		\dots
		(\tau_{1} \tau_{2} \tau_{1})\\
	&= (\tau_2 \tau_1 \tau_2)
		\dots
		(\tau_{l-2} \tau_{l-1} \tau_{l-2})
		\Delta_{\g_{l-1},\g_{l}}
		(\tau_{l-1} \tau_{l-2} \tau_{l-1})
		\dots
		(\tau_{1} \tau_{2} \tau_{1}).
	\end{align*}
Our goal is to factor this expression into a sequence of $\Delta$-twists involving~$\g_1$. 
The key observation is that
	\begin{align*}
	&(\tau_{l-2} \tau_{l-1} \tau_{l-2})
	\Delta_{\g_{l-1},\g_{l}}
	(\tau_{l-1} \tau_{l-2} \tau_{l-1}) \\
	=&
	(\tau_{l-2} \tau_{l-1} \tau_{l-2})
	\Delta_{\g_{l-1},\g_{l}}
	(\tau_{l-2} \tau_{l-1} \tau_{l-2}) \\
	=&
	(\tau_{l-2} \tau_{l-1} \tau_{l-2})
	\Delta_{\g_{l-1},\g_{l}}
	(\tau_{l-2} \tau_{l-1} \tau_{l-2})\inv
	\Delta_{\g_{l-2},\g_{l-1}}
	 \\
	=&
	\Delta_{\tau_{l-2} \tau_{l-1} \tau_{l-2}(\g_{l-1}),\tau_{l-2} \tau_{l-1} \tau_{l-2}(\g_{l})}
	\Delta_{\g_{l-2},\g_{l-1}} \\
	=&
	\Delta_{\g_{l-2},\tau_{l-2} \tau_{l-1} \tau_{l-2}(\g_{l})}
	\Delta_{\g_{l-2},\g_{l-1}}.
	\end{align*}
Applying this repeatedly we eventually obtain
	\begin{equation*}
	\mu=
		\Delta_{\g_{1},\delta_{l}}
		\Delta_{\g_{1},\delta_{l-1}}
		\dots
		\Delta_{\g_{1},\delta_{2}}
	\end{equation*}
where $\delta_k:=(\tau_1 \tau_2 \tau_1)\dots(\tau_{k-2} \tau_{k-1} \tau_{k-2}) (\g_k)$.
Hence, the monodromy of~$D\SD$ is trivial by Corollary~\ref{T:kernel of surgery homomorphism}.

If~$\G$ was a closed circuit to begin with so that~$\SD=(\S,\G)$ is a surface diagram, then one can show that~$Z_{D\SD}$ closes off to~$DZ_\SD=Z_\SD\cup_\del\overline{Z_\SD}$, the double of~$Z_\SD$, whence the name.
\end{example}

\subsection{Drawing Kirby diagrams}
	\label{S:Kirby diagrams}
In this section we show how to translate surface diagrams into Kirby diagrams of the associated \swfs. 
For the necessary background we refer the reader to~\cite{GS}. Throughout, we use Akbulut's \emph{dotted circle notation} for 1-handles to avoid ambiguities for framing coefficients.

\subsubsection{Descending \swfs}
Let $w\colon Z\ra D^2$ be a descending \swf of genus~$g$ with surface diagram $\SD=(\S_g;\g_1,\dots,\g_c)$.
Recall that the associated handle decomposition of~$Z$ is obtained from (some handle decomposition of) $\S_g\times D^2$ by attaching 2-handles along~$\g_i\subset\S_g\times\set{\theta_i}$ with respect to the fiber framing where $\theta_1,\dots,\theta_c\in S^1$ are ordered according to the orientation on~$S^1$. 
So in order to draw a Kirby diagram for~$Z$ we need to find a diagram for~$\S\times D^2$ in which the fibers of the boundary should be as clearly visible as possible.

A convenient choice is the diagram shown in Figure~\ref{F:surface times disk 1} which is induced from the obvious handle decomposition of~$\S_g$ with one 0-handle, $2g$~1-handles and one 2-handle.
One fiber of $\S_g\times S^1$, which we identify with~$\S_g$, is clearly visible and the canonical generators $a_1,b_1,\dots,a_g,b_g$ for~$H_1(\S_g)$ are also indicated. We have chosen the orientations such that~$\scp{a_i,b_i}_{\S_g}=1$. 
Another advantage of this picture is that the fiber framing agrees with the blackboard framing. One minor drawback is that the picture does not immediately show \emph{all} fibers of $\S_g\times S^1$ but only an interval worth of them (just thicken the surface a little).
However, this is actually enough for our purposes since we only need the fibers over the interval~$[\theta_1,\theta_c]\subset S^1$.
To get the orientations right we require that the orientation of the fiber agrees with the standard orientation of the plane and, according to the ``fiber first convention'', the positive $S^1$-direction points out of the plane.
	\begin{figure}[h]
	\includegraphics{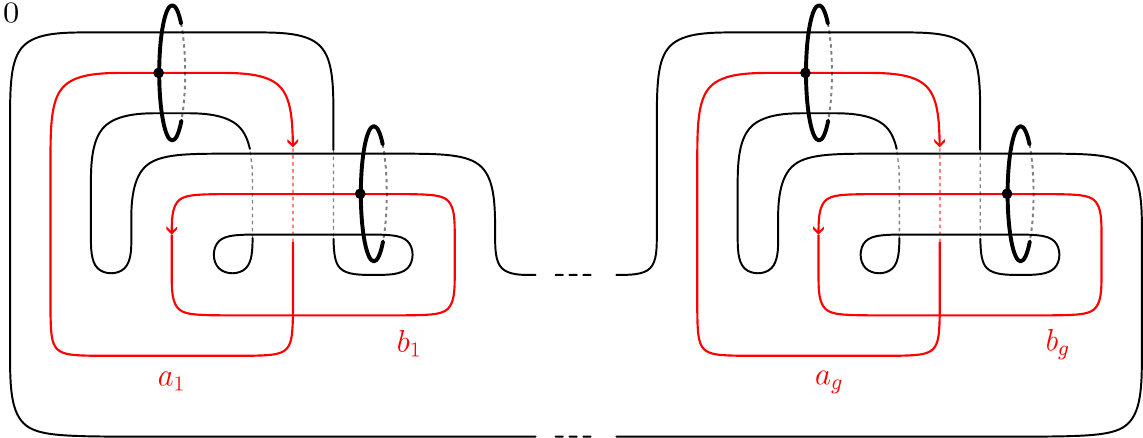}
	\caption{A diagram for $\S_g\times D^2$ where fiber and blackboard framing agree. The red curves show a basis for $H_1(\S_g)$.}
	\label{F:surface times disk 1}
	\end{figure}

\smallskip
With this understood, it is easy to locate the attaching curves of the fold handles in the diagram and it remains to determine their framing coefficients. More generally, we can describe the linking form of the link corresponding to the fold handles.
It should be no surprise that the framing and linking information in the diagram depends on our choice of the handle decomposition for~$\S_g$.

Let~$\g\subset\S_g$ be a \scc. After choosing an orientation its homology class $[\g]\in H_1(\S)$ can be expressed as
	\begin{equation*}
	[\g]=\sum_{i=1}^g \big( n_{a_i}(\g)\,a_i + n_{b_i}(\g)\,b_i \big).
	\end{equation*}
We identify~$\S_g$ with~$\S_g\times\set{0}$ and, by a slight abuse of notation, we continue to denote the canonical push-off of~$\g$ to~$\S_g\times\set{z}$, $z\in D^2$, by~$\g$.

\begin{lemma}\label{T:framings and linking}
For a \scc $\g\subset\S_g\times\set{\theta}$, $\theta\in [\theta_1,\theta_c]\subset S^1$, the framing coefficient of the fiber framing in Figure~\ref{F:surface times disk 1} is given by 
	\begin{equation}\label{E:fiber framing}
	\mathrm{fr}(\g)=  \sum_{i=1}^g n_{a_i}(\g) n_{b_i}(\g).
	\end{equation}
Furthermore, if $\g\subset\S_g\times\set{\theta}$ and $\g'\subset\S_g\times\set{\theta'}$, $\theta,\theta'\in[\theta_1,\theta_c]$, are two oriented \sccs, then their linking number in Figure~\ref{F:surface times disk 1} is
	\begin{equation}\label{E:linking}
	\begin{split}
	\mathrm{lk}(\g,\g')
		=&\frac{1}{2} \mathrm{sgn}(\theta-\theta') \big\langle \g,\g' \big\rangle \\
		 &+\frac{1}{2} \sum_{i=1}^g \big[ n_{a_i}(\g)n_{b_i}(\g')+n_{a_i}(\g')n_{b_i}(\g) \big]
	\end{split}
	\end{equation}
where $\scp{\g,\g'}$ is the algebraic intersection number of $\g$~and~$\g'$ in~$\S_g$ and $\mathrm{sgn}$ denotes the sign of a real number\footnote{To avoid any confusion, we use the convention that $\mathrm{sgn}(0)=0$.}.
\end{lemma}
\begin{proof}
First observe that $\g\subset \S_g\times\set{\theta}$ can be isotoped off the 2-handle to be completely visible in Figure~\ref{F:surface times disk 1} and, since the fiber framing and blackboard framing agree, its framing coefficient is given by its writhe in the diagram, i.e.~the signed count of crossings with some chosen orientation.
From the way the diagram is drawn it is clear that each crossing is caused by~$\g$ running over~$a_i$ \emph{and}~$b_i$ for some~$i$ 
and that their signed sum is given by the right hand side of~\eqref{E:fiber framing}.

The statement about linking numbers follows from a similar count of crossings. 
Recall that the linking number of two oriented knots can be computed from any link diagram as half of the signed number of crossings. 
The second term on the right hand side of~\eqref{E:linking} arises just as above. However, the first term deserves some explanation. 
Each (transverse) intersection point of $\g$~and~$\g'$ in~$\S_g$ contributes a crossing in the diagram. Now, the sign of the crossing depends one two things: the sign of the intersection point and the information which strand is on top in the diagram. 
From Figure~\ref{F:intersections and crossings} we see that the contribution of each crossing is exactly as in~\eqref{E:fiber framing}.
	\begin{figure}
	\includegraphics[scale=.8]{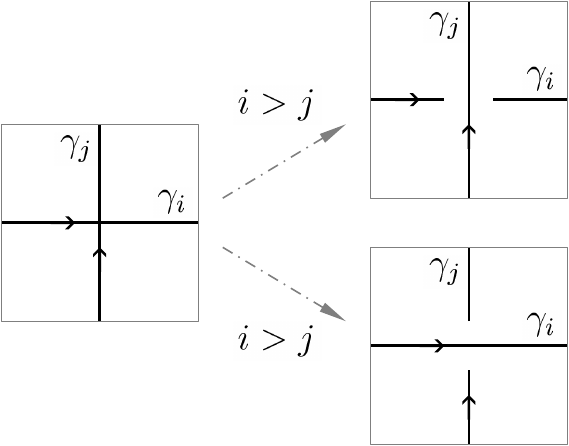}
	\caption{An intersection in a surface diagram and its crossing in the Kirby diagram.}
	\label{F:intersections and crossings}
	\end{figure}
\end{proof}

\begin{remark}\label{R:intersection form}
Formula~\ref{E:linking} can be used to obtain a description of the intersection form of the 4-manifold~$Z_\SD$ described by a surface diagram~$\SD$ which only uses the data in~$\SD$.
Moreover, since~\ref{E:linking} only depends on the homology classes of the curves in~$\SD$, so do the intersection form and, in particular, the signature of~$Z_\SD$.
We will return to this observation in a future publication.
\end{remark}

The diagrams of \swfs derived from Figure~\ref{F:surface times disk 1} are good for abstract reasoning, however, in practice it is convenient to start with a cleaner diagram for~$\S_g\times D^2$ such as the one shown in Figure~\ref{F:surface times disk 2}. 
In this picture, the fiber appears as the boundary sum of regular \nbhds of the basis curves~$\set{a'_i,b_i}_{i=1}^{g}$ which, in turn, appear as meridians to the dotted circles.
	\begin{figure}[h]
	\centering\includegraphics{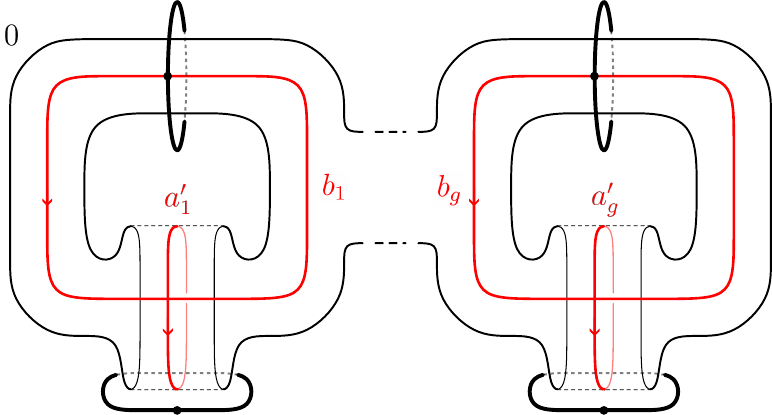}
	\caption{A cleaner diagram of $\S_g\times D^2$.}
	\label{F:surface times disk 2}
	\end{figure}
The framing coefficient of the fiber framing for \sccs on a fiber in Figure~\ref{F:surface times disk 2} can be computed as follows.
It is not hard to see that Figure~\ref{F:surface times disk 2} is obtained from Figure~\ref{F:surface times disk 1} by a sequence of 1-handle slides and an isotopy of the 2-handle and vice versa. 
Note that these moves do not change the framing coefficients of any other 2-handles that might have been around.
Moreover, during the moves, the $b$-curves remain fixed, while the $a$-curves undergo some changes. 
When pulling~$a_i'$ in Figure~\ref{F:surface times disk 2} back to Figure~\ref{F:surface times disk 1} one obtains a representative for the element 
	\begin{equation*}
	[a_1,b_1]\ast\dots\ast[a_{i-1},b_{i-1}]\ast a_i\in\pi_1(\S_g)
	\end{equation*}
where $[x,y]=xyx\inv y\inv$. 
The important observation is that while this curve is not isotopic to~$a_i$ it does represent the same homology class. 
As a consequence, formula~\eqref{E:fiber framing} can be used for Figure~\ref{F:surface times disk 2} with~$a_i$ replaced by~$a_i'$.

\subsubsection{Closing off and the last 2-handle}
	\label{S:the last 2-handle}

Recall that our motivation comes from Williams' theorem that all closed, oriented 4-manifolds admit \swfs over~$S^2$. 
We have seen that these can be described (up to equivalence) by surface diagrams with trivial monodromy and we have already mentioned that it is in general not easy to check whether the monodromy of a given surface diagram is trivial. But the situation is even worse.
Say that we know for some reason that a given surface diagram has trivial monodromy and let us also assume that the genus is at least three so that there are no gluing ambiguities. Even in this case it is not clear at all how the surface diagram encodes the information to complete the Kirby diagram.

To be more precise, let~$w\colon X\ra S^2$ be a \swf with surface diagram~$\SD$. Let~$\nu\S_-$ be a \nbhd of a lower genus fiber and let~$Z:=X\setminus\nu\S_-$.
Then~$w$ restricts to a descending \swf on~$Z$ and~$\del Z$ can be identified with $\S_-\times S^1$ so that~$\SD$ must have trivial monodromy. 
We can draw a Kirby diagram for~$Z$ as described in the previous section and to complete it to a diagram for~$X$ we have to understand how to glue~$\nu\S_-$ back in.

We can choose a handle decomposition for $\nu\S_-$ with one 0-handle, $2g(\S_-)$ 1-handles and one 2-handle. 
Turning this upside down results in a relative handle decomposition on~$\del Z\cong \S_-\times S^1$ with one 2-handle, $2g(\S_-)$ 3-handles and a 4-handle.
The general theory tells us that the 3-~and 4-handles attach in a standard way once we know how to attach the 2-handle. 
Unfortunately, it turns out to be rather difficult to locate this \emph{last 2-handle} in the Kirby diagram for~$Z$.

Our knowledge about the last 2-handle is a priori limited to the following observation. 
If we identify~$\nu\S_-$ with~$\S_-\times D^2$, then the attaching curve of the last 2-handle corresponds to~$\set{p}\times \del D^2$ for some~$p\in\S_-$. 
In particular, we see that it must be attached along a section of the boundary fibration~$(\del Z,w)$. 

\begin{remark}\label{R:closing off by Kirby moves}
Given a surface diagram~$\SD$ with trivial monodromy, there is a general method for finding possible last 2-handles for~$Z_\SD$ which is not very conceptual but still useful in some situations.\footnote{Compare Chapter~8.2 in~\cite{GS} (p.~299f) for the Lefschetz case.}
One considers a Kirby diagram for~$Z_\SD$ as a surgery diagram for~$\del Z_\SD$ and performs  (3-dimensional) Kirby moves until the fibration structure is clearly visible as~$\S_-\times S^1$.
In such a diagram it is easy to locate the possible attaching curves for last 2-handles. One can then pull back these curves to the original diagram by undoing the moves and dragging the curves along.
\end{remark}

Just as in the Lefschetz case, the situation becomes easier if one knows that~$Z_\SD$ can be closed off to a fibration over~$S^2$ which admits a section. The proof of the following lemma is the same as in the Lefschetz case and we refer the reader to~\cite{GS}.

\begin{lemma}\label{T:closing off with section}
Let $w\colon X\ra S^2$ be a \swf with surface diagram~$\SD$.
If~$w$ admits a section of self-intersection~$k$, then the last two handle appears in the diagram for~$Z_\SD$ as a $k$-framed meridian of the 2-handle corresponding to the fiber.
Furthermore, if~$\SD$ is a surface diagram and a meridian as above can be used to attach the last 2-handle, then the corresponding \swf admits a section of self-intersection~$k$.
\end{lemma}

In order to illustrate Remark~\ref{R:closing off by Kirby moves} and Lemma~\ref{T:closing off with section} as well as our method of drawing Kirby diagrams we give an example which is also a warm up for the next section.

\begin{example}\label{eg:Kirby diagram example}
Let $a,b\subset\S_g$ be a geometrically dual pair of \sccs. We claim that $\SD=(\S_g;a,\tau_b(a),b)$ is a surface diagram for $\S_{g-1}\times S^2\#\CPbar$.
	\begin{figure}
	\includegraphics{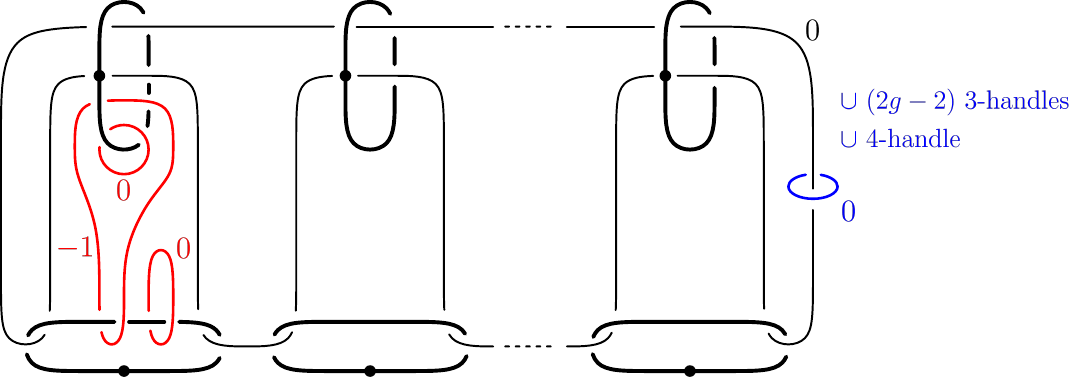}
	\caption{Manifolds with surface diagram~$(\S_g;a,\tau_b(a),b)$}
	\label{F:blow-up preview}
	\end{figure}
We can assume that~$a$ and~$b$ are the standard generators~$a_1$ and~$b_1$ in Figure~\ref{F:surface times disk 2} and Figure~\ref{F:blow-up preview} shows the final Kirby diagram. 
In order to see how we got there let us first ignore all the blue components. What is left is just the Kirby diagram for~$Z_\SD$. The framings on the fold handles can either be computed using Lemma~\ref{T:framings and linking} (together with Proposition~\ref{T:Picard-Lefschetz formula}) or by hand\footnote{The curve is simple enough to draw a parallel push-off in the fiber direction and compute the linking number}.
We now perform the obvious handle moves: using the meridians to the two 1-handles on the left we first unlink the $-1$-framed fold handle (corresponding to~$\tau_b(a)$) to obtain a $-1$-framed unknot isolated from the rest of the diagram, then we unlink the black 2-handle (corresponding to the fiber) and finally cancel the 1-handles and their meridians. 
Obviously, the thus obtained diagram shows $\S_{g-1}\times D^2 \# \CPbar$ and the boundary is clearly visible as~$\S_{g-1}\times S^1$. Moreover, it is easy to see that the last 2-handle can be attached along a 0-framed meridian to the fiber 2-handle and the resulting manifold is~$\S_{g-1}\times S^2 \# \CPbar$ as claimed. 
Finally, since we attached the last 2-handle in a region that was not affected by the Kirby moves it will not change when we undo the moves again and we arrive at Figure~\ref{F:blow-up preview}. Lemma~\ref{T:closing off with section} then tells us that the corresponding \swf will have a section of self-intersection zero.

Note that for~$g\geq 3$ the way we have attached the last 2-handle is unique. In the lower genus cases there are more options. However, in any case one will end up with a blow-up of some surface bundle over~$S^2$.
\end{example}

\subsection{Relation to broken Lefschetz fibrations}
	\label{S:relation to BLFs}

Let $w\colon X\ra B$ be a \swf. After trading all the cusps for Lefschetz singularities by applying Lekili's unsinking modification we obtain a broken Lefschetz fibration
	\begin{equation*}
	\beta_w\colon X \ra B
	\end{equation*}
with one round singularity, smoothly embedded in the base, and all its Lefschetz points on the higher genus side. 
If the base is the sphere or the disk, then~$\beta_w$ is a \emph{simplified broken Lefschetz fibration} in the sense of~\cite{Baykur2} and thus induces another handle decomposition of~$X$.

\smallskip
In order to relate these two handle decompositions, let us briefly review how a handle decomposition is obtained from a simplified broken Lefschetz fibration~$\beta\colon X\ra B$.
Much in the spirit of \swfs one chooses a reference point in the higher genus region together with a collection of disjointly embedded arcs $L_1,\dots,L_k,R\subset B$, where $k$ is the number of Lefschetz singularities, emanating from the reference point such that each~$L_i$ ends in a Lefschetz point and~$R$ passes through the round singularity once. Such a system of arcs is known as a \emph{Hurwitz system} for~$\beta$.
The arcs in a Hurwitz system then give rise to \sccs in the reference fiber~$\S$ to which we shall refer to as the \emph{Lefschetz vanishing cycles}~$\lambda_1,\dots,\lambda_k\subset\S$ and the \emph{round vanishing cycle}~$\rho$.
A handle decomposition of~$X$ is then given as follows:
\begin{itemize}
	\item Start with $\S\times D^2$
	\item Going around $S^1$ attach a \emph{Lefschetz handle} along the~$\lambda_i$ pushed off into fibers over~$S^1$, i.e.~2-handles with framing~$-1$ \wrt the fiber framing
	\item Attach a \emph{round 2-handle} along $\rho$
\end{itemize}
The round 2-handle decomposes into a 2-handle and a 3-handle such that the 3-handle goes over the 2-handle geometrically twice and the 2-handle is attached along~$\rho$ \wrt the fiber framing. (For more details see~\cite{Baykur2}.)

\smallskip
Now let $w\colon X\ra B$ be a \swf and let $\beta_w$ be the associated simplified broken Lefschetz fibration.
Given a reference system $\mcR=\set{R_i}$ for~$w$ with associated surface diagram $(\Sigma,\G)$ there is a canonical Hurwitz system for~$\beta_w$. Since the unsinking homotopy is supported near the cusps we can assume that the nothing happens around the reference paths. Now observe that the arcs~$R_i$ cut the higher genus region into triangles each containing a single Lefschetz singularity of~$\beta_w$. Thus, up to isotopy, there is a unique arc $L_i$ in the triangle bounded by $R_i$ and $R_{i+1}$ going from the reference fiber to the Lefschetz singularity and for the round singularity we take the arc $R=R_1$. 
According to Lekili~\cite{Lekili}, the vanishing cycles of~$\beta_w$ \wrt this Hurwitz system are given by
	\begin{equation*}
	\lambda_i=\tau_{\g_i}(\g_{i+1}) \quad\text{and}\quad \rho=\g_1.
	\end{equation*}
We can go from the handle decomposition induced by~$\beta_w$ to the one induced by~$w$ using the following handlebody interpretation of the (un-)sinking deformation.
	\begin{figure}
	\includegraphics{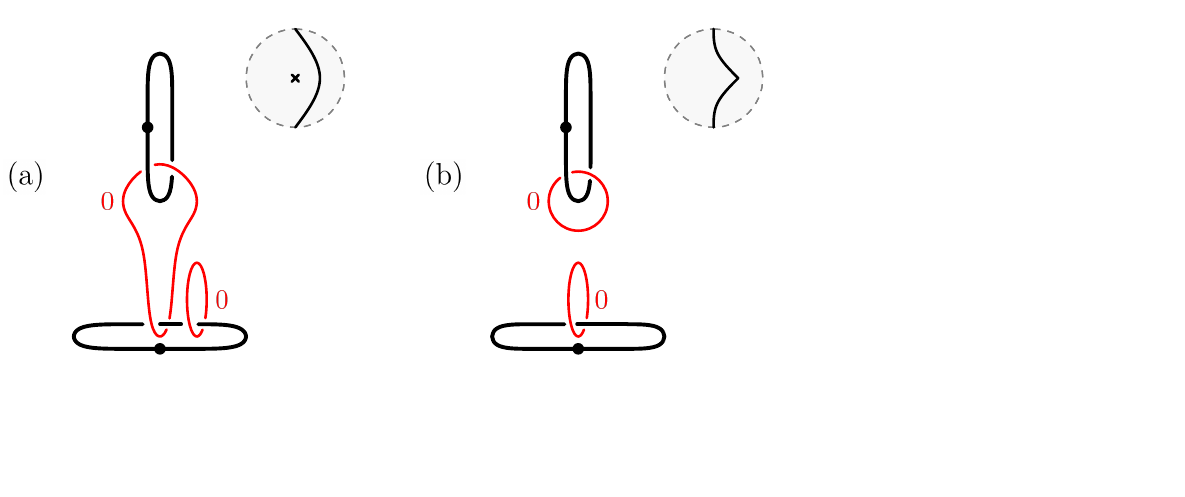}
	\caption{A Lefschetz singularity (a) before and (b) after sinking.}
	\label{F:SWF vs BLF}
	\end{figure}

Assume that we have a Lefschetz singularity next to a fold arc that is \emph{sinkable},~i.e. the Lefschetz and fold vanishing cycles intersect in one point. (In other words, it is the resulting of unsinking a cusp.)
In terms of handle decompositions the situation before and after the sinking process is locally described in Figure~\ref{F:SWF vs BLF}.\footnote{These handle decompositions have already appeared in a disguised form in~\cite{Lekili}.}
(The Lefschetz 2-handle in~(a) is the one that goes over both 1-handles. One readily checks that it is correctly framed.)
Clearly, both pictures describe a 4-ball and they are related by an obvious 2-handle slide.
Indeed, to go from~(a) to~(b) one has to slide the Lefschetz handle over the fold handle in such a way that it unlinks from the lower 1-handle.
Note that his handle slide is compatible with the fibration structures in the sense that the attaching curves stay on the fibers. Moreover, it mysteriously adjusts the framings exactly as needed. 

\begin{remark}\label{R:wrinkling via handles}
Although the handle slide described above seems to be a correct interpretation of Lekili's (un-)sinking deformation it is a priori not obvious why this should be true.
In fact, the deformation is a combination of wrinkling, merging and flipping (see~\cite{Lekili}, Figure~8) and does not seem very atomic. On the other hand, the handle slide is an atomic modification of the handlebodies. It would be interesting to see a 1-parameter family of Morse functions associated with the (un-)sinking deformation that would exhibit the handle slide.
\end{remark}

This shows that, if we start we the handle decomposition of~$\beta_w$, then sliding~$\lambda_1$ over~$\rho=\g_1$ produces a fiber framed attaching curve~$\lambda_1'$ which is isotopic to~$\g_2$. 
Successively sliding~$\lambda_{i}$ over~$\lambda_{i-1}'\sim\g_i$ results in fiber framed attaching curves~$\lambda_{i}'$ isotopic to~$\g_{i+1}$.
Altogether we end up with fiber framed curves~$\lambda_1',\dots,\lambda_c',\rho$. 
The final observation is that~~$\lambda_c'$ is isotopic to~$\rho=\g_1$ and can be unlinked and isolated from the rest of the diagram to form a zero framed unknot which cancels the 3-handle coming from the round singularity.
What we are left with is the decomposition associated to~$w$.

\section{Substitutions}
	\label{S:substitutions}
Let $\SD=(\S,\G)$ be a surface diagram and let $\Lambda$ be a subcircuit of $\G$. If $\Lambda'$ is any circuit that starts and ends with the same curves as~$\Lambda$, then we can build a new surface diagram $(\S,\G')$ where $\G'$ is obtained by replacing $\Lambda$ with $\Lambda'$. 
We call this operation a \emph{substitution of type $(\Lambda|\Lambda')$}\footnote{Similar substitution techniques for Lefschetz fibrations are studied in~\cites{Endo1,Endo2}.}.

Passing to the associated \swfs one can ask how such a substitution affects the total spaces. In the following we will treat two instances in which this question can be answered. Our main tool are the handle decompositions exhibited in the previous section.

\smallskip
Let $Z$ be a compact 4-manifold, possibly with nonempty boundary. Recall that the \emph{blow-up} of~$Z$ is the connected sum of~$Z$ with either~$\CPbar$ or~$\CP$ (taken in the interior of~$Z$). Moreover, the \emph{sum stabilization} of~$Z$ usually means the connected sum with~$S^2\times S^2$. We will be slightly more general and also allow connected sums with~$\CP\#\CPbar$, the twisted $S^2$-bundle over~$S^2$. For convenience, we let 
	\begin{equation*}
	\mathbb{S}_k:=
	\begin{cases}
	S^2\times S^2, & \text{$k$ even}\\
	\CP\#\CPbar, & \text{$k$ odd}
	\end{cases}
	\end{equation*}
and note that $\mathbb{S}_k$ is described by the $(0,k)$-framed Hopf link.

\begin{lemma}[Blow-ups and sum stabilizations]\label{T:blow-up lemma}
Let $\SD=(\S,\G)$ be a surface diagram and let $\SD'$ be obtained from $\SD$ by a substitution of type 
	\begin{equation}\label{E:blow-up configuration}
	\big( a,b \,|\, a,\tau_b^{\pm1}(a),b \big).
	\end{equation}
Furthermore, let $\SD''$ be obtained by a substitution of type 
	\begin{equation}\label{E:stabilization configuration}
	\big( a,b \,|\, a,b,\tau_b^{k}(a),b \big).
	\end{equation}
Then $Z_{\SD'}$ is diffeomorphic to the blow-up~$Z_\SD\#\mp\CP$ and~$Z_{\SD''}$ is diffeomorphic to the sum stabilization~$Z_\SD\#\mathbb{S}_{-k}$.
\end{lemma}

Of course, any substitution is reversible so that whenever a surface diagram contains a configuration of the form $(a,\tau_b^{\pm1}(a),b)$ or $(a,b,\tau_b^{k}(a),b)$ the associated 4-manifold must be a blow-up or sum stabilization, respectively. We will call these \emph{blow-up} (resp. \emph{sum stabilization}) \emph{configurations}.

\begin{proof}
By switching we can assume that $\G=(\dots,a,b)$ and thus $\G'(\dots,a,\tau_b^{\pm1}(a),b)$ and $\G''=(\dots,a,b,\tau_b^{k}(a),b)$. 
Figure~\ref{F:blow-up lemma} shows the relevant parts of the handle decompositions of the associated 4-manifolds. The shaded ribbons indicate the regions that contain all the other fold handles. 
Note that the curves $a$~and~$b$ appear as 0-framed meridians to the dotted circles.
	\begin{figure}
	\includegraphics{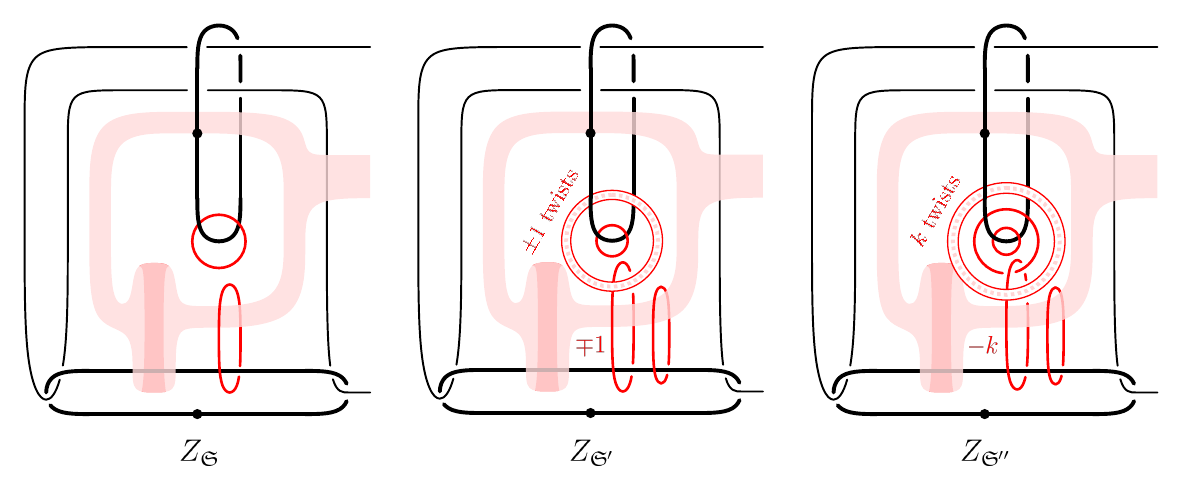}
	\caption{The relevant parts of the handle decompositions of~$Z_\SD$, $Z_{\SD'}$ and~$Z_{\SD''}$. All 2-handles without framing coefficient are $0$-framed.}
	\label{F:blow-up lemma}
	\end{figure}

In the case of $Z_{\SD'}$ we can use the meridians to unlink the curve corresponding to~$\tau_b^\pm(a)$ resulting in an unknot with framing $\mp 1$ which is isolated from the rest of the diagram. Furthermore, the rest of the diagram agrees with the diagram for~$Z_\SD$ and the claim follows.

The argument for $Z_{\SD''}$ is almost the same. Again, by sliding over the meridians we can isolate the curves corresponding to~$b$ and~$\tau_b^k(a)$ from the rest of the diagram. This time we obtain a $(0,-k)$-framed Hopf link which represents a copy of~$\mathbb{S}_{-k}$.
\end{proof}

\begin{proposition}\label{T:blow-up closed}
Let $\SD$, $\SD'$ and $\SD''$ be as in Lemma~\ref{T:blow-up lemma}. 
\begin{enumerate}
	\item All three diagrams have the same monodromy. 
	\item If $\SD$ has trivial monodromy so that~$Z_\SD$ closes off to a closed 4-manifold~$X$, then $Z_{\SD'}$ (resp.~$Z_{\SD'}$) closes off to~$X\#\mp\CP$ (resp.~$X\#\mathbb{S}_k$).
	\item Any closed 4-manifold obtained from~$\SD'$ (resp.~$\SD''$) is a blow-up (resp.~sum-stabilization) of a manifold obtained from~$\SD$.
\end{enumerate}
\end{proposition}
\begin{proof}
The first statement follows directly from Lemma~\ref{T:blow-up lemma} since connected sums with closed manifold (taken in the interior) do not change the boundary.

For the other statements, observe that if one knows how to apply the method from Remark~\ref{R:closing off by Kirby moves} for~$\SD$, then one also knows it for~$\SD'$ (resp.~$\SD''$) and vice versa. 
\end{proof}

Another instance where a substitution corresponds to a well known cut-and-paste operation has been observed by Hayano (\cite{HayanoR2},~Lemma~6.13). 
Assume that a surface diagram~$\SD$ contains a curve~$c\subset \S$. If $d\subset\S$ is geometrically dual to~$c$, then one can perform a substitution of type~$(c|c,d,c)$ and Hayano shows that if~$\SD'$ denotes the resulting surface diagram, then~$Z_{\SD'}$ is obtained from~$Z_\SD$ by a surgery on the curve $\delta\subset\S\subset Z_\SD$ with respect to its \emph{fiber framing}, i.e.~the framing induced by the its canonical framing in~$\S$ together with the framing of~$\S$ in~$Z_\SD$ as a regular fiber of~$w_\SD\colon Z_\SD\ra D^2$.

One immediately notices that our sum-stabilization substitution is a special case of this construction. However, it also leads the way to the following minor generalization of the surgery substitution which captures not only the fiber framed surgery but also the one with the opposite framing.
\begin{lemma}\label{T:Hayano surgery}
Let $\SD$ and $\SD'$ be two surface diagram with the same underlying surface~$\S$ and let $c,d\subset\S$ be a geometrically dual pair of \sccs.
If~$\SD'$ is obtained from~$\SD$ by a substitution of type~$(c|c,\tau_c^k(d),c)$, then~$Z_{\SD'}$ is obtained from~$Z_\SD$ by a surgery on~$d\subset\S\subset X$ \wrt the fiber framing when~$k$ is even and the opposite framing when~$k$ is odd.
\end{lemma}
\begin{proof}
As in Hayano's proof, it is enough to work in a \nbhd of $c\cup d$ which we can assume to be a punctured torus. 
Using our handle decomposition instead of the ones from broken Lefschetz fibrations, the effect of Hayano's surgery substitution, i.e.~the case when~$k=0$, looks as in Figure~\ref{F:Hayano surgery} where~$c$ (resp.~$d$) appears as the meridian of the upper (resp.~lower) 1-handle.
	\begin{figure}
	\includegraphics{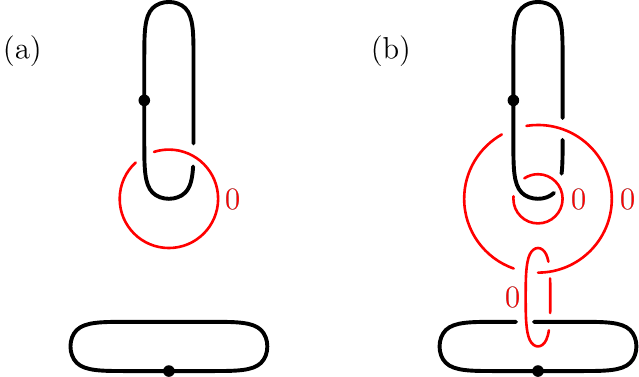}
	\caption{Hayano's surgery substitution: \nbhds with (a)~vanishing cycle~$c$ and (b)~vanishing cycles~$(c,d,c)$. }
	\label{F:Hayano surgery}
	\end{figure}
To obtain the other even cases, observe that in Figure~\ref{F:Hayano surgery}(b) we can slide the 2-handle corresponding to~$d$ once over each 2-handle corresponding to~$c$ in the same direction.
Depending on the direction this changes the framing coefficient by $\pm2$ and one readily checks that the resulting curve diagram shows a \nbhd with vanishing cycles $(c,\tau_c^{\mp2}(d),c)$. Repeating this trick one can obtain all configurations with even~$k$ and they will all describe the fiber framed surgery on~$d$.

As shown in~\cite[Example~8.4.6]{GS} the surgery with the opposite framing can be realized by inserting a pair of a Lefschetz vanishing cycle and an achiral Lefschetz vanishing cycle which are both parallel to~$d$.
But Figure~\ref{F:opposite surgery} shows that the result is the same as a substitution of type~$(c|c,\tau_c\inv(d),c)$ which corresponds to~$k=-1$.
	\begin{figure}
	\includegraphics{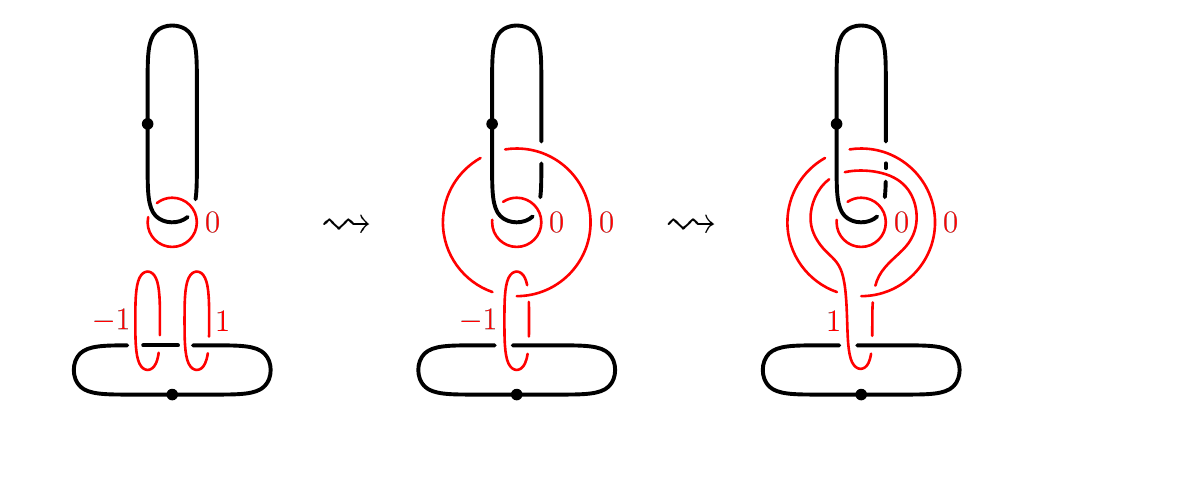}
	\caption{Surgery with the opposite framing.}
	\label{F:opposite surgery}
	\end{figure}
Moreover, the arguments for shifting the value of~$k$ by multiples of~$2$ works just as in the fiber framed case.
\end{proof}

Using the above lemma the sum-stabilization can be interpreted as performing surgery on a null-homotopic curve with either of its framing. Indeed, as~$d$ one takes one of the adjacent vanishing cycles of~$c$ in~$\SD$ which is clearly null-homotopic in~$Z_\SD$.

\smallskip
It would be interesting to interpret other cut-and-paste operations on 4-manifolds as substitutions in surface diagrams. 
For example, it is reasonable to expect such an interpretation for certain rational blow downs which can be described in terms of Lefschetz fibrations (see~\cite{Endo2}).
However, we will settle for blow-ups and sum-stabilizations in this paper.

\section{Manifolds with genus 1 \swfs}
	\label{S:genus 1 classification}

In this section we prove Theorem~\ref{T:genus 1 classification, intro}. 
Our strategy is to use Proposition~\ref{T:blow-up closed} to construct some genus~$1$ \swfs and then show that this construction gives all such fibrations.

\smallskip
We begin with the construction of genus~1 \swfs over~$S^2$. As before, we denote by~$\mathbb{S}_k$ the closed 4-manifolds described by the $(0,k)$-framed Hopf link and we define a family of manifolds
	\begin{equation}\label{E:genus 1 manifold list}
	X_{klmn}=\mathbb{S}_k \# l(S^2\times S^2) \# m\CP \# n\CPbar ,
	\quad k\in\set{0,1},\;\, l,m,n\geq0.
	\end{equation}
Note that these are precisely the manifolds in Theorem~\ref{T:genus 1 classification, intro}.
Recall that $\mathbb{S}_k$ is an $S^2$-bundle over~$S^2$. By performing a birth on a suitable bundle projection~$\mathbb{S}_k\ra S^2$ we obtain a \swf with two cusps. We can then use Lemma~\ref{T:blow-up lemma} to add the other summands at will.
Thus, in order to prove Theorem~\ref{T:genus 1 classification, intro}, it remains to show the following.
\begin{proposition}\label{T:genus 1 classification}
Let $w\colon X\ra S^2$ be a \swf of genus~1. 
Then~$X$ is diffeomorphic to some~$X_{klmn}$ described in~\eqref{E:genus 1 manifold list}.
\end{proposition}
\begin{remark}\label{R:reformulation of classification}
The reason for our small reformulation of Theorem~\ref{T:genus 1 classification, intro} is that, while the original formulation is cleaner, the new one is much more in tune with the structure of the proof.
\end{remark}

The key to the proof of Proposition~\ref{T:genus 1 classification} is the simple nature of simple closed curves on the torus. 
Indeed, the two well known facts that two oriented \sccs on the torus are isotopic if and only if they are homologous and that the (absolute value of the) algebraic and geometric intersection numbers agree allow us to transfer the whole discussion of genus~1 surface diagrams into the homology group~$H_1(T^2)\cong\Z\oplus\Z$ simply by choosing orientations on the curves.
Building on this observation we obtain the following result about the structure of genus~1 surface diagrams.

\begin{lemma}\label{T:genus 1 circuits}
Any closed circuit on the torus of length at least three contains blow-up or sum stabilization configurations (as described in Lemma~\ref{T:blow-up lemma}).
\end{lemma}
\begin{proof}
Let $\G=(\g_1,\dots,\g_c)$ be a (not necessarily closed) circuit on the torus of length $c\geq3$. As usual, we choose an arbitrary orientation on $\g_1$ and orient the remaining curves by requiring that $\scp{\g_i,\g_{i+1}}=1$ for $i<c$ so that we can consider each $\g_i$ as an element of $H_1(T^2)$.

We first observe that, since any two adjacent curves in~$\G$ algebraically dual, they form a basis of~$H_1(T^2)$. In particular, for $i\geq3$ we can write
	\begin{equation*}
	\g_i=k_i\g_{i-1}-\g_{i-2},\quad k_i\in\Z
	\end{equation*}
where the coefficient of $\g_{i-2}$ determined by our convention that $\scp{\g_{i-1},\g_i}=1$. This shows that if we denote by $\sigma_i:=\scp{\g_1,\g_i}$ the algebraic intersection number between~$\g_1$ and $\g_i$, then we have $\sigma_1=0$, $\sigma_2=1$ and the recursion formula
	\begin{equation}\label{E:intersection recursion}
	\sigma_{i}=k_i\sigma_{i-1}-\sigma_{i-2}
	\end{equation}
holds for $i\geq3$. At this point we note that $\G$ is closed if and only if $|\sigma_c|=1$.

\smallskip
We claim that if $|k_i|\geq2$ for all $i\geq3$, then $|\sigma_{i+1}|>|\sigma_{i+1}|$ for all~$i$. This follows inductively since $|\sigma_2|>|\sigma_1|$ and from~\eqref{E:intersection recursion} we get
	\begin{align*}
	|\sigma_{i+1}| 
		&=|k_{i+1}\sigma_i-\sigma_{i-1}| \\
		& \geq \big| |k_{i+1}||\sigma_i|-|\sigma_{i-1}| \big| \\
		& = |k_{i+1}||\sigma_i|-|\sigma_{i-1}| > |\sigma_i|
	\end{align*}
where we have used the reverse triangle inequality, the induction hypothesis and the assumption that $|k_{i+1}|\geq2$.
As a consequence, we see that if~$\G$ is closed, then we must have~$|k_i|\leq1$ for some~$i\geq3$. 

Assume first that $k_i=\pm1$. For the sake of a cleaner notation we momentarily rename the relevant curves to
	\begin{equation}\label{E:detecting a blow-up}
	(\g_{i-2},\g_{i-1},\g_i)=:(a,\xi,b).
	\end{equation}
By assumption, we have $b=\pm\xi-a$ and thus $\xi=\pm(a+b)$ and the orientation convention shows that~$\scp{a,b}=\pm1$.
By invoking the Picard-Lefschetz formula (Proposition~\ref{T:Picard-Lefschetz formula}) we obtain
	\begin{align*}
	\tau_a^{\pm1}(b)
	&=b\pm\scp{a,b}a \\
	&=a+b \\
	&= \pm\xi
	\end{align*}
which, after forgetting the orientations again, reveals the excerpt of~$\G$ shown in~\eqref{E:detecting a blow-up} as a blow-up configuration.

A similar argument exhibits a sum-stabilization configuration in the remaining case when~$k_i=0$. The details are left to the reader.
\end{proof}

The proof of Proposition~\ref{T:genus 1 classification}, and thus of Theorem~\ref{T:genus 1 classification, intro} is now very easy.

\begin{proof}[Proof of Proposition~\ref{T:genus 1 classification}]
Any genus one \swf over~$S^2$ can is obtained by closing off a manifold~$Z_\SD$ associated to a surface diagram~$\SD=(T^2,\G)$.
Moreover, any such diagram~$\SD$ can be closed off since the mapping class group of the lower genus fiber is trivial.
By Lemma~\ref{T:genus 1 circuits} and Proposition~\ref{T:blow-up closed}~(3) we can successively split off summands of the form~$\pm\CP$ and~$\mathbb{S}_k$ until the remaining surface diagram, say~$\SD_0$ has a circuit of length two.
It is easy to see that $Z_{\SD_0}$ is the trivial disk bundle~$S^2\times D^2$. 
(Either by drawing a Kirby diagram or by observing that any \swf with two cusps is homotopic to a bundle projection.) 
Thus there are exactly two ways to close off the fibration, producing a summand of the form~$\mathbb{S}_0\cong S^2\times S^2$ or~$\mathbb{S}_1\cong\CP\#\CPbar$.
\end{proof}

\section{Concluding remarks}
	\label{S:concluding remarks}
The theory of simple wrinkled fibrations and surface diagrams is still in a very early stage and at this point it raises more questions then it provides answers.
We would like to take the opportunity to point out some of the major problems in the subject as well as to indicate some further developments.

\subsection{Closed 4-manifolds}
The ultimate goal is to use surface diagrams to study \emph{closed} 4-manifolds. Unfortunately, it turns out that most surface diagrams do \emph{not} describe closed manifolds since they have non-trivial monodromy and 
it is usually a hard problem to determine whether a given surface diagram has trivial monodromy.
The following is thus of great interest.

\begin{problem}\label{P:trivial monodromy conditions}
Find at least necessary conditions for a surface diagram to have trivial monodromy that are easier to check.
\end{problem}

The next major problem was already mentioned in Section~\ref{S:the last 2-handle}. If a surface diagram of sufficiently high genus is known to have trivial monodromy, then it determines a unique closed 4-manifold together with a \swf over~$S^2$ by closing off the associated fibration over the disk. 
However, the way that the surface diagram encodes the closing off information is too implicit for practical purposes. For example, by simply looking at the surface diagram it not at all clear how to answer the following very reasonable questions about the corresponding \swf over~$S^2$:
\begin{itemize}
	\item Does the fibration have a section?
	\item What can be said about the homology class of the fiber? (Is it trivial, primitive, torsion,... ?)
	\item What is the fundamental group, homology, etc. of the total space? 
\end{itemize}
What is missing is one more piece of information which is roughly the (framed) attaching curve of the last 2-handle. 
One can reformulate this issue in terms of mapping class groups (see~\cite{HayanoR2}, for example) 

\begin{problem}\label{P:closing off}
Find a practical method to determine the missing piece of information from a surface diagram with trivial monodromy. 
\end{problem}

\subsection{Higher genus fibrations}

The fact that any (achiral) Lefschetz fibration can be turned into a \swf of one genus higher suggests the philosophy that \swfs of a fixed genus might behave similarly as (achiral) Lefschetz fibrations of one genus lower.

This analogy works rather well for the lowest possible fiber genera. Indeed, our result about genus one \swfs looks very similar to the (rather trivial) classification of genus zero (achiral) Lefschetz fibrations, the latter being blow-ups of either~$S^2\times S^2$ or~$\CP\#\CPbar$.

Following this train of thought one might hope to be able to say something useful about the classification of genus two \swfs over~$S^2$ but one should expect to be lost as soon as the genus is three or higher. 
However, it is nonetheless conceivable that part of the classification scheme that worked in the genus one case might carry over to higher genus fibrations, as we will now explain

\smallskip
Let $\SD=(\S;\g_1,\dots,\g_l)$ be a surface diagram and assume that for some~$2<k<l$ the curve~$\g_k$ is geometrically dual to~$\g_1$. Then there is an obvious way to decompose~$\SD$ into the two smaller surface diagrams~$(\S;\g_1,\dots,\g_k)$ and~$(\S;\g_1,\g_k,\dots,\g_l)$.
Repeating this process we eventually obtain a decomposition of~$\SD$ into a collection of surface diagram with the property that no pair of non-adjacent curves has geometric intersection number one. 
Let us call such a surface diagram \emph{irreducible}. 

In terms of the \swf associated to~$\SD$ the above decomposition of~$\SD$ should correspond to merging the fold arcs that induce~$\g_1$ and~$\g_k$. The result is a wrinkled fibration that naturally decomposes as a \emph{boundary fiber sum} of the two \swfs associated to the parts of the decomposition of~$\SD$. 

This suggests that any descending \swf over the disk naturally decomposes into a boundary fiber sum of \emph{irreducible} fibrations where we call a \swf irreducible if its surface diagram is irreducible. 
Consequently, the classification of descending \swfs splits into two parts: the classification of irreducible fibrations and understanding the effect of boundary fiber sums. 

\smallskip
The genus one classification fits into this scheme as follows.
Our arguments show that the only irreducible surface diagrams of genus one are given by the blow-up configurations $(a,\tau_a^{\pm1}(b),b)$ and the sum-stabilization configurations $(a,b,\tau_b^k(a),b)$ for~$k\neq1$. 
Using the handle decompositions it is easy to identify the corresponding manifolds.
(They are the connected sum of $S^2\times D^2$ with either $\pm\CP$, $S^2\times S^2$ or $\CP\#\CPbar$.)
Furthermore, the boundary fiber sums are performed along spheres and are thus easy to understand.

\smallskip
Making these arguments precise requires an understanding of the effect of merging folds and cusps on surface diagrams.

\subsection{Uniqueness of surface diagrams}

Given the fact that all closed 4-manifolds can be described by surface diagrams, it is natural to ask for a set of moves to relate different surface diagrams that describe the same manifold, similar to the situation of 3-manifolds and Heegaard diagrams.

A first step in this direction was taken by Williams~\cite{Williams2} who relates the surface diagrams of homotopic \swfs over~$S^2$ of genus at least three.
He shows that any two homotopic \swfs can be connected by a special homotopy that is made up of four basic building blocks. These building blocks are simple enough to understand their effect on the initial surface diagram (see also the recent work of Hayano~\cite{HayanoR2}).

So far this is completely analogous to the 3-dimensional context.
A new phenomenon in the 4-dimensional context is that two \swfs on a given 4-manifold are not necessarily homotopic. 
The structure of the set $\pi^2(X):=[X,S^2]$ of homotopy classes of maps from a closed 4-manifold to the 2-sphere -- also known as the \emph{second cohomotopy set} of~$X$ -- is described in~\cite{cohomotopy} (see also the references therein).
Our results show that an equivalence class of surface diagrams for~$X$ determines an orbit of the action of the diffeomorphism group of~$X$ on~$\pi^2(X)$.
This action is usually 
neither trivial\footnote{For example the diffeomorphism of~$S^2\times S^2$ that interchanges the two factors also interchanges the projections onto the factors which are easily seen not to be homotopic.}
nor transitive\footnote{This follows from the fact that the diffeomorphism action on~$H_2(X)$ preserves divisibility.}.
Consequently, reparametrizing a surface diagram can change the homotopy class of its \swf but one cannot expect to obtain all homotopy classes in this way. 

A general method for relating broken fibrations in different homotopy classes is the \emph{projection move} mentioned in~\cite{Williams1} but it is not at all obvious how to interpret this procedure in terms of surface diagrams. 
Altogether, the problem of relating surface diagram with non-homotopic fibrations is still wide open.

% --> BIBLIOGRAPHY --> %

\end{document}